\documentclass[11pt]{article}

\usepackage[a4paper, margin=1in]{geometry}
\usepackage[open,openlevel=1]{bookmark}
\usepackage{pgfplots}
\usepackage{amsmath}
\usepackage{tikz}
\usetikzlibrary{cd}
\usetikzlibrary{arrows}
\usetikzlibrary{decorations.markings}
\usetikzlibrary{arrows.meta}
\usetikzlibrary{plotmarks}
\usetikzlibrary{decorations.pathmorphing}
\usetikzlibrary{decorations.pathreplacing}
\usetikzlibrary{positioning}
\usetikzlibrary{fadings}
\usetikzlibrary{intersections}
\usetikzlibrary{knots}
\usepackage[parfill]{parskip}
\usepackage[font=small,labelfont=bf, margin=0.3cm]{caption}
\usepackage{tabu}
\usepackage{arydshln}
\usepackage{changepage}
\usepackage{sidecap}
\usepackage{multirow}
\usepackage{wrapfig}
\usepackage{amssymb}
\usepackage{fancyvrb}
\usepackage{color}
\usepackage{mdframed}
\usepackage[inline,shortlabels]{enumitem}
\usepackage{adjustbox}
\usepackage{graphicx}
\usepackage{amsthm}
\usepackage{xcolor}
\usepackage{mdframed}
\usepackage{framed}
\usepackage{tablefootnote}
\usepackage{mathrsfs}
\usepackage{tensor}
\usepackage{xparse}
\usepackage{comment}
\usepackage{hyperref}

\DeclareFontFamily{U}{mathx}{\hyphenchar\font45}
\DeclareFontShape{U}{mathx}{m}{n}{
      <5> <6> <7> <8> <9> <10>
      <10.95> <12> <14.4> <17.28> <20.74> <24.88>
      mathx10
      }{}
\DeclareSymbolFont{mathx}{U}{mathx}{m}{n}
\DeclareFontSubstitution{U}{mathx}{m}{n}
\DeclareMathAccent{\widecheck}{0}{mathx}{"71}
\DeclareMathAccent{\wideparen}{0}{mathx}{"75}

\colorlet{shadecolor}{lightgray!60}

\makeatletter
\def\th@plain{%
  \thm@notefont{}
  \itshape 
}
\def\th@definition{%
  \thm@notefont{}
  \normalfont 
}
\makeatother

\renewenvironment{proof}
{\begin{adjustwidth}{0.35cm}{0.35cm}\textit{Proof.}}
{\end{adjustwidth}}

\theoremstyle{definition}

\newtheorem{theorem}{Theorem}[section]
\newtheorem{corollary}{Corollary}[theorem]
\newtheorem{definition}[theorem]{Definition}
\newtheorem{lemma}[theorem]{Lemma}
\newtheorem{proposition}[theorem]{Proposition}

\allowdisplaybreaks

\input{insbox}

\newmdenv[
  topline=false,
  bottomline=false,
  rightline=false,
  skipabove=\topsep,
  skipbelow=\topsep,
  leftmargin=3pt,
  linewidth=1pt,
  linecolor=gray,
  nobreak=false,
]{lr}

\sidecaptionvpos{figure}{c}

\makeatletter
\def\cl{\@ifnextchar{\bgroup}{\@with}{\@without}}
\def\@with#1{\overline{#1}}
\def\@without{\bar}
\makeatother

\makeatletter
\def\hatt{\@ifnextchar{\bgroup}{\@withaz}{\@withoutaz}}
\def\@withaz#1{\widehat{#1}}
\def\@withoutaz{\hat}
\makeatother

\makeatletter
\def\checkk{\@ifnextchar{\bgroup}{\@withb}{\@withoutb}}
\def\@withb#1{\widecheck{#1}}
\def\@withoutb{\check}
\makeatother

\newcommand{\E}{\mathcal{E}}

\newcommand{\R} {\mathbb{R}}

\newcommand{\K} {\mathbb{K}}

\newcommand{\abs}[1]{\left\lvert #1\right\rvert}
\newcommand{\brac}[1]{\left( #1\right)}

\newcommand{\mat}[1]{\begin{pmatrix}#1\end{pmatrix}}

\newcommand{\eva}[1]{\left.#1\right|}

\newcommand*{\myb}[1]{\emph{#1}}

\newcommand*{\mybb}[1]{\underline{\smash{\textbf{#1}}}}

\definecolor{light-gray}{gray}{0.8}
\definecolor{dark-gray}{gray}{0.2}

\tikzset{->/.style = {decoration={markings,
                                  mark=at position 1 with {\arrow[scale=1.2]{latex'}}},
                      postaction={decorate}}}
\tikzset{<-/.style = {decoration={markings,
                                  mark=at position 0 with {\arrowreversed[scale=1.2]{latex'}}},
                      postaction={decorate}}}
\tikzset{<->/.style = {decoration={markings,
                                   mark=at position 0 with {\arrowreversed[scale=1.2]{latex'}},
                                   mark=at position 1 with {\arrow[scale=1.2]{latex'}}},
                       postaction={decorate}}}
\tikzset{->-/.style = {decoration={markings,
                                   mark=at position #1 with {\arrow[scale=1.5]{latex'}}},
                       postaction={decorate}}}
\tikzset{-<-/.style = {decoration={markings,
                                   mark=at position #1 with {\arrowreversed[scale=1.5]{latex'}}},
                       postaction={decorate}}}
                       
\pgfarrowsdeclarecombine{twolatex'}{twolatex'}{latex'}{latex'}{latex'}{latex'}
                       
\tikzset{->>/.style = {decoration={markings,
                                  mark=at position 1 with {\arrow[scale=1.2]{latex'}}},
                      postaction={decorate}}}
\tikzset{<<-/.style = {decoration={markings,
                                  mark=at position 1 with {\arrowreversed[scale=1.2]{twolatex'}}},
                      postaction={decorate}}}
\tikzset{<<->>/.style = {decoration={markings,
                                   mark=at position 0 with {\arrowreversed[scale=1.2]{twolatex'}},
                                   mark=at position 1 with {\arrow[scale=1.2]{twolatex'}}},
                       postaction={decorate}}}
\tikzset{->>-/.style = {decoration={markings,
                                   mark=at position #1 with {\arrow[scale=1.5]{twolatex'}}},
                       postaction={decorate}}}
\tikzset{-<<-/.style = {decoration={markings,
                                   mark=at position #1 with {\arrowreversed[scale=1.5]{twolatex'}}},
                       postaction={decorate}}}

\tikzset{circ/.style = {fill, circle, inner sep = 0, minimum size = 4}}
\tikzset{mstate/.style={circle, draw, gray, text=black, minimum width=0.5cm}}

\newlength{\currentparskip}
\newenvironment{minipageparskip}[1]
  {\setlength{\currentparskip}{\parskip}
   \begin{minipage}{#1}
   \setlength{\parskip}{\currentparskip}
  }
  {\end{minipage}}
  
\newenvironment{minipageparskipb}[1]
  {\setlength{\currentparskip}{\parskip}
   \begin{minipage}[b]{#1}
   \setlength{\parskip}{\currentparskip}
  }
  {\end{minipage}}
  
  \newenvironment{minipar1}[1]
  {\setlength{\currentparskip}{\parskip}
   \begin{minipage}[t]{#1}
   \setlength{\parskip}{\currentparskip}
  }
  {\end{minipage}}

\renewcommand{\qed}{\nobreak \ifvmode \relax \else
      \ifdim\lastskip<1.5em \hskip-\lastskip
      \hskip1.5em plus0em minus0.5em \fi \nobreak
      \vrule height0.75em width0.5em depth0.25em\fi}

\newcommand\Item[1][]{%
  \vspace{2mm}
  \ifx\relax#1\relax  \item \else \item[#1] \fi
  \abovedisplayskip=0pt\abovedisplayshortskip=0pt~\vspace*{-\baselineskip}\vspace*{-1.5mm}}
  
  \newcommand\iitem[1][]{%
  \vspace{2mm}
  \ifx\relax#1\relax  \item \else \item[#1] \fi
  \abovedisplayskip=0pt\abovedisplayshortskip=0pt~\vspace*{-\baselineskip}}

\setlist[enumerate,1]{label=\roman*., leftmargin=8mm}
\setlist[itemize,1]{label=\textbullet, leftmargin=6mm}

\DeclareMathOperator{\supp}{supp}

\renewcommand{\Re}{\text{Re}}

\newcommand{\mb}{\mathbf}

\newcommand{\ph}{\,\cdot\,}

\newcommand{\<}{\langle}
\renewcommand{\>}{\rangle}

\newcommand{\st} {\text{ s.t.}\:}

\newcommand{\mathsarray}{\arraycolsep=1.4pt\def\arraystretch{1.3}}

\DeclareMathOperator{\dist}{dist}

\newcommand{\lpc}{\Delta}
\newcommand{\grad}{\nabla}

\title{\huge\textbf{Linear Stability of liquid Lane-Emden stars}}
\date{}

\begin{document}



\begin{center}
{\LARGE \textbf{Linear Stability of liquid Lane-Emden stars}}

\vspace{3mm}

King Ming Lam\footnote{Department of Mathematics, University College London, London WC1E 6XA, UK. Email: \href{mailto:king.lam.19@ucl.ac.uk}{king.lam.19@ucl.ac.uk}}
\end{center}

\vspace{7mm}

\begin{abstract}
\noindent
We establish various qualitative properties of liquid Lane-Emden stars in $\R^d$, including bounds for its density profile $\rho$ and radius $R$. Using them we prove that against radial perturbations, the liquid Lane-Emden stars are linearly stable when $\gamma\geq 2(d-1)/d$; linearly stable when $\gamma<2(d-1)/d$ for stars with small relative central density $\rho(0)-\rho(R)$; and linearly unstable when $\gamma<2(d-1)/d$ for stars with large central density. Such dependence on central density is not seen in the gaseous Lane-Emden stars.
\end{abstract}

\section{Introduction}

A classical model for stars is given by the \myb{Euler-Poisson system}, as a lump of compressible gas or liquid surrounded by vacuum, subjected to forces created by its own pressure and gravity. The fluid is modelled by the \myb{compressible Euler equations}, which model perfect fluids with no heat conduction and no viscosity. It is given by
\begin{align}
\rho{D\mb u\over Dt}&=-\grad p+\rho\mb g\label{momentum equation}\\
\partial_t\rho+\grad\cdot(\rho\mb u)&=0\label{continuity equation}
\end{align}
where $\rho$ is the fluid density, $\mb u$ the fluid velocity, $\mb g$ the body accelerations acting on the fluid (e.g.\ gravity, inertial accelerations, electric field acceleration etc.), $p$ the fluid pressure. All these quantities are a function of space $\mb x$ (in $\R^d$) and time $t$. Here ${D\over Dt}:=\partial_t+\mb u\cdot\grad$ denotes the material derivative. Equation (\ref{momentum equation}) is the so-called \myb{momentum equation} which describes the forces on the fluid particles and conservation of momentum. Equation (\ref{continuity equation}) is the so-called \myb{continuity equation} which describes the conservation of mass.

We take $\mb g$ to be the Newtonian gravity
\[\mb g=\grad\phi\]
where $\phi$ is defined by 
\begin{align*}
\lpc\phi=4\pi\rho\qquad\text{and}\qquad\lim_{\|\mb x\|\to\infty}\phi(t,\mb x)=0.
\end{align*}
Here the boundary condition at infinity $\lim_{\|\mb x\|\to\infty}\phi(t,\mb x)=0$ ensures that we have a unique solution for $\phi$, and furthermore we require this boundary condition to be zero since we want to model an isolated star where no energy is coming from infinity. This case we have describes a \myb{self-gravitating star}, subject to gravity induced by its own mass as well as force coming from its own pressure. 

The lump of liquid does not have a fixed shape, its shape and volume can change - the pressure of the liquid will try to force the liquid to occupy more space/volume and expend the domain, while any gravitational force will try to do the opposite. As a result we have a \myb{free surface problem} where the domain $\Omega(t)=\supp\rho$ changes with time, being deformed by the fluid motion itself. We have the following boundary conditions:
\begin{enumerate}[a)]
\item the pressure on the boundary matches that of the vacuum outside, i.e. $p=0$ on $\partial\Omega$;
\item the normal velocity at which the boundary changes is equal to $\mb u\cdot\mb n$ at any point on the boundary, where $\mb n$ denotes the outward unit normal vector to $\partial\Omega$.
\end{enumerate}

We now have $d+1$ equations for our fluid (the scalar continuity equation (\ref{continuity equation}) plus the vector momentum equation (\ref{momentum equation})), but $d+2$ unknowns ($\rho$, $p$ and $\mb u$). To close the system, we need to specify an \myb{equation of state}. The classical model which we will consider is that of a \myb{barotropic fluid} where density is a function of pressure only, wherein we have a equation of state $p=P(\rho)$ that relates the pressure to the density. Under this we shall consider the classical case of \myb{polytropic fluid} where $p$ is related to $\rho$ via
\begin{align}
p=\begin{cases}K\rho^\gamma&\text{ for gasous stars}\\K\rho^\gamma-C&\text{ for liquid stars.}\end{cases}\label{PES}
\end{align}
for some constant $K,C>0$ and $2\geq\gamma\geq1$. Note that under (\ref{PES}), as pressure is 0 on the boundary $\partial\Omega$, the density of the fluid will be $0$ on the boundary for the gas case, and $(C/K)^{1/\gamma}$ on the boundary for the liquid case. Without loss of generality, we can take $K=C=1$ as these constants can be normalised to be 1. The constant for the power $\gamma$, called the \myb{adiabatic index}, does makes a difference however as it defines the fundamental relationship between $p$ and $\rho$. With respect to the natural scaling of the system, the parameter $\gamma$ plays a crucial role as natural criticality parameter, with different behaviour to be expected from difference values of $\gamma$. We will prove result for the liquid case, but to do so we will also need to understand the gaseous case.

\subsection{History and Background}

The Euler equations were first described by Euler in 1757, but even in the absence of vacuum it is a highly non-trivial problem. In particular, the Euler equations is prone to singularity/shock formation even from quite regular data, see for example \cite{CHRISTODOULOU} by Christodoulou and \cite{SIDERIS} by Sideris. This made the study of existence and behaviour and global solutions non-trivial. In the presence of vacuum the problem is even trickier due to the degeneracy of the boundary and its movements. When the density is bounded away from zero, the system is strictly hyperbolic and the classical theory of hyperbolic systems for local existence applies: \cite{MAJDA}. In the presence of the vacuum, we have singularity or discontinuity across the boundary and this no longer applies, and so even local existence is non-trivial, with various local existence theory only establish recently.

The Euler-Poisson equations with the polytropic equation of state has long been used to model stellar structure and evolution by astrophysics such as Chandrasekhar \cite{Chandrasekhar}, Shapiro and Teukolsky \cite{Shapiro Teukolsky}. Despite and because of its simplicity, it is a good approximation for certain regions of various types of stars, and so it serves as the classical model of stars. In particular, many equations of state that are of astrophysical interest behave like polytropes for low and for high pressures. One of the most important class of solutions to the Euler-Poisson equation with the polytropic equation of state are the time independent steady state solutions known as the \myb{Lane–Emden stars} (for any $\gamma\in[1,2]$). It is the classical model of a spherically symmetric stationary steady star where pressure and gravity are in perfect balance. In 1980 astrophysicists Goldreich and Weber \cite{GOLDREICH WEBER} found a special class of expanding and collapsing solutions for $\gamma=4/3$, the mass critical index where the natural scaling preserves the mass. This is another important class of solutions as it models stellar collapse and expansion such as supernova expansion. Importantly, both the Lane–Emden stars and the Goldreich-Weber stars in the gaseous case has $w\sim\dist(\ph,\Omega^c)$ near the vacuum boundary. Therefore, given these special solutions of physical interest, it is important to have an existence theory for the Euler and Euler-Poisson equation that is compatible and includes this particular boundary behaviour.

Recently progress on local existence theory in this area has been made. In this formulation, one has a moving domain $\Omega(t)$, treating $\partial\Omega(t)$ as an unknown that evolves with the flow, prescribing the so-called \myb{physical boundary condition} $0<\mb n\cdot\grad w<\infty$ on the boundary where $\mb n$ is the outward normal of the boundary, i.e. $w\sim\dist(\ph,\Omega^c)$ near the boundary. For the Euler equations in the gaseous case in this setting, Coutand and Shkoller \cite{Coutand Shkoller}, and Jang and Masmoudi \cite{Jang Masmoudi} have recently independently proved local existence. A different proof based on Eulerian approach (rather than using Lagrangian coordinates as previous work) was later given by Ifrim and Tataru \cite{Ifrim Tataru}. Extensions of local existence result to the Euler-Poisson equations has been done by using the fact that the gravity term is of lower order - see \cite{Gu Lei} by Gu and Lei, and \cite{Hadzic Jang} by Had\v{z}i\'c and Jang. For the liquid case, local-existence for the Euler system is established by Lindblad \cite{Lindblad} using Nash–Moser construction; Trakhinin \cite{Trakhinin} using the theory of symmetric hyperbolic systems; and Coutand, Hole and Shkoller \cite{Coutand Hole Shkoller} using the vanishing viscosity method with a parabolic regularization together with some time-differentiated a priori estimates. Local existence for the liquid Euler-Poisson equation has been proven by Ginsberg, Lindblad and Luo \cite{Ginsberg Lindblad Luo}. In the relativistic setting, this has been done by Oliynyk \cite{Oliynyk} and Miao, Shahshahani and Wu \cite{Miao Shahshahani Wu}.

Since already the Euler equation alone is prone to singularity/shocks which made global-in-time behaviour non-trivial, when coupled with an additional attractive gravitational force, we expect the Euler-Poisson to be even more prone to singularity or blow ups. Indeed, it has been shown that solutions could blow up in a finite time - \cite{Deng Xiang Yang}\cite{Makino}. As a result there are not many global-in-time existence and uniqueness results for the Euler-Poisson system outside of the aforementioned special solutions. Due to their physical significance, it is important to understand what happens in the vicinity of these special solutions, whether they are linearly and non-linearly stable.

On the problem of stability of gaseous Lane–Emden stars it is known \cite{Lin}\cite{Jang Makino} that the Lane–Emden stars in $\R^3$ are linearly unstable when $1<\gamma<4/3$ and linearly stable when $4/3\leq\gamma<2$. Recently, nonlinear instability in the range $6/5\leq\gamma<4/3$ has been proven in \cite{Jang 1}\cite{Jang 2} by Jang. Nonlinear stability in the range $4/3<\gamma<2$ has not been fully proven yet and remains open, but under the assumption that a global-in-time solution exists, nonlinear stability in the range $4/3<\gamma<2$ has been shown by variation arguments, see \cite{Rein} by Rein, and \cite{Luo Smoller} by Luo and Smoller. In the critical case $\gamma=4/3$ the Lane-Emden star is nonlinearly unstable despite the conditional linear stability - in fact the family of expanding/collapsing Goldreich-Weber stars that exist for this $\gamma$ can get arbitrarily close to the Lane-Emden star. Nonlinear stability of the expanding gaseous Goldreich-Weber stars against radially symmetric perturbations was proven by Had\v{z}i\'c and Jang \cite{Hadzic Jang 2}, and recently generalised to non-radial perturbations by the author together with Had\v{z}i\'c and Jang in \cite{Hadzic Jang Lam}.

Making use of our linear (in)stability analysis in this work, very recently Hao and Miao have proven in \cite{Hao Miao} non-linear instability of liquid Lane–Emden stars in $\R^3$ with large central density in the regime $1\leq\gamma<4/3$.

\subsection{Statement of main result and plan for the paper}

Despite the linear stability result on the gaseous Lane–Emden stars, the question of linear stability of \emph{liquid} Lane–Emden stars has not been established to our knowledge. The linear stability of the gaseous Lane–Emden stars does not depend on its central density. However, notably, we found that in the liquid case, the stability of the Lane–Emden stars does depend on its central density.

\begin{definition}\label{def 1}
For a liquid Lane–Emden stars with density profile $\rho$, we say that it has \emph{small relative central density} if $\rho(0)-1$ is close to 0.
\end{definition}

In this paper we will prove the following linear stability result for the liquid Lane–Emden stars.

\begin{theorem}\label{liquid Lane–Emden stars stability}
Assume $d<10$. The liquid Lane–Emden stars are, against radial perturbations, 
\begin{enumerate}
\item linearly stable when $\gamma\geq 2(d-1)/d$;
\item linearly stable when $\gamma<2(d-1)/d$ for stars with small relative central density (see Definition \ref{def 1});
\item linearly unstable when $\gamma<2(d-1)/d$ for stars with large central density.
\end{enumerate}
\end{theorem}

 In this sense it is analogous to relativistic stars described by the Einstein–Euler equations, where it was found in \cite{Hadzic Lin Rein} by Had\v{z}i\'c, Lin and Rein that strongly relativistic steady stars (those with very large central density) are unstable. It turns out the imposition of a liquid boundary creates extrema points on the mass-radius plot when $\gamma<2(d-1)/d$ so that they end up similar to those in the relativistic case, where a turning point principle proven by Had\v{z}i\'c, Lin and Rein \cite{Hadzic Lin Rein}\cite{Hadzic Lin} dictates that such extrema points introduces unstable modes (see diagram). The two problems share similarity in such a way that the methods of dynamical system utilised in \cite{Hadzic Lin Rein} can be adapted for use in our study of liquid Lane–Emden stars - we use that to prove part iii of our theorem.

In the section 2 of the paper, we will establish various qualitative properties of the liquid Lane–Emden stars, which are steady states solutions of the Euler-Poisson system, including bounds for its density profile and radius. And using these results, in section 3 we will prove theorem \ref{liquid Lane–Emden stars stability}, where we will prove the existence and non-existence of unstable modes using the variational formulation.

\section{Liquid Lane–Emden stars and its properties}

The Lane–Emden stars are steady states solutions of the following Euler-Poisson system
\begin{align}
\rho{D\mb u\over Dt}+\grad p+\rho\grad\phi&=\mb 0\label{momentum equation 2}\\
\partial_t\rho+\grad\cdot(\rho\mb u)&=0\label{continuity equation 2}
\end{align}
where
\begin{align*}
\lpc\phi=4\pi\rho\qquad\text{ with }\qquad\lim_{\|\mb x\|\to\infty}\phi(t,\mb x)=0
\end{align*}
and $p$ is either
\[p=\begin{cases}\rho^\gamma&\text{ for gasous stars}\\\rho^\gamma-1&\text{ for liquid stars}\end{cases}\]
with boundary condition
\[\mathsarray
\begin{array}{rl}
p|_{\partial\Omega(t)}&=0\\
\grad w\cdot\mb n|_{\partial\Omega(0)}&<0
\end{array}\qquad\qquad\text{ where }\qquad\qquad\begin{array}{rl}
\Omega(t)&=\supp\rho(t)\subseteq\R^d,\qquad d\geq 3\\
w&=\rho^{\gamma-1}
\end{array}\]
and $\mb n$ is the outward normal of $\partial\Omega(t)$.

Seeking steady states reduces the system to
\begin{align}
\grad p+\rho\grad\phi=\mb 0\qquad\text {with boundary conditions}\qquad p|_{\partial\Omega}=0\qquad\text{ and }\qquad\grad w\cdot\mb n|_{\partial\Omega}<0\label{steady states equation}
\end{align}
where all the variables are independence from $t$. We will refer to the equation $\grad p+\rho\grad\phi=\mb 0$ as the \myb{steady state equation}.

We will mainly study the liquid stars, but to do so will will need to study properties of gaseous stars.

\subsection{The spherically symmetric steady state equation as an ODE}

The ODE that defines the Lane–Emden stars, which are spherically symmetric steady states, is given by the following proposition.

\begin{proposition}
The density $\rho$ for the Lane–Emden stars satisfies:
\begin{enumerate}
\item when $\gamma>1$,
\begin{align}
{1\over r^{d-1}}{d\over dr}\brac{r^{d-1}{dw\over dr}}=-4\pi{\gamma-1\over\gamma}w^\alpha.\label{enthalpy equation}
\end{align}
or equivalently
\begin{align}
{dw\over dr}=-4\pi{\gamma-1\over\gamma}{1\over r^{d-1}}\int_0^ry^{d-1}w(y)^\alpha dy\label{*}
\end{align}
where $w=\rho^{\gamma-1}$.
\item when $\gamma=1$,
\begin{align}
{1\over r^{d-1}}{d\over dr}\brac{r^{d-1}{dh\over dr}}=-4\pi e^h.\label{$h$-equation}
\end{align}
or equivalently
\begin{align}
{dh\over dr}=-4\pi{1\over r^{d-1}}\int_0^ry^{d-1}e^{h(y)}dy\label{**}
\end{align}
where $h=\ln\rho$.
\end{enumerate}
\end{proposition}
\begin{proof}
\begin{enumerate}
\item Assuming $\gamma>1$, we can divide the steady state equation (\ref{steady states equation})
\[\mb 0=\grad p+\rho\grad\phi=\gamma\rho^{\gamma-1}\grad\rho+\rho\grad\phi\]
by $\rho$ and rewrite it as
\[\mb 0=\gamma\rho^{\gamma-2}\grad\rho+\grad\phi={\gamma\over\gamma-1}\grad\rho^{\gamma-1}+\grad\phi.\]
Apply $\grad\cdot$ to both sides we get
\[\mb 0={\gamma\over\gamma-1}\lpc\rho^{\gamma-1}+4\pi\rho={\gamma\over\gamma-1}\lpc w+4\pi w^{\alpha}\qquad\text{where}\qquad\alpha=(\gamma-1)^{-1}.\]

We seek radially symmetric solutions depending only on $r=\|\mb x\|$, then using
\[\mathsarray
\begin{array}{rl}
\grad r^\alpha&=\alpha r^{\alpha-1}\hat{\mb r}\\
\grad\cdot(r^\alpha\hat{\mb r})&=(\alpha-1+d)r^{\alpha-1}
\end{array}\qquad\qquad\text{ where }\qquad\qquad
\hat{\mb r}=\mb x/\|\mb x\|
\]
we deduced that
\begin{align*}
\lpc w(r)&=\grad\cdot(w'(r)\grad r)=w''(r)(\grad r)\cdot(\grad r)+w'(r)\lpc r\\
&=w''(r)+(d-1)r^{-1}w'(r)={1\over r^{d-1}}{d\over dr}\brac{r^{d-1}{dw\over dr}}.
\end{align*}
So we want to seek solutions of
\[{1\over r^{d-1}}{d\over dr}\brac{r^{d-1}{dw\over dr}}=-4\pi{\gamma-1\over\gamma}w^\alpha.\]

\item The steady state equation is now $\grad\rho+\rho\grad\phi=\mb 0$, or equivalently $\grad(\ln\rho)+\grad\phi=\mb 0$. Taking divergence on both sides we get
\[0=\lpc(\ln\rho)+4\pi\rho={\rho''\over\rho}-{(\rho')^2\over\rho^2}+(d-1){1\over r}{\rho'\over \rho}+4\pi\rho.\]
Let $h=\ln\rho$, then the equation becomes
\[\lpc h+4\pi e^h=0.\]
This almost the same equation as the case when $\gamma>1$ with variable $h$ in place of $w$, except the non-linear term is $4\pi e^h$ instead of $4\pi(\gamma-1)\gamma^{-1}w^\alpha$. So instead of $(*)$ we have
\[{dh\over dr}=-4\pi{1\over r^{d-1}}\int_0^ry^{d-1}e^{h(y)}dy.\qed\]
\end{enumerate}
\end{proof}

\subsection{Basic properties of Lane–Emden stars}

From now on we will refer to spherically symmetry steady states as simply steady states, and all our functions will only depend on the one variable $r$.

\begin{theorem}
Let $\gamma\geq 1$. For every $\rho_0>0$, the steady state equation admits a unique solution $\rho\geq 0$ such that $\rho(0)=\rho_0$. The interval of existence $[0,R)$ is such that either $R=\infty$ or $\lim_{r\to R}\rho(r)=0$.
\end{theorem}
\begin{proof}
This is a know standard result, we included the proof in the appendix for completeness.\qed
\end{proof}

In fact, due to the following decay estimate, we must have $\lim_{r\to R}\rho(r)=0$ even if $R=\infty$. This is a elementary decay estimate that we will be using in many places later on.

\begin{lemma}[Decay estimates]
Suppose $\rho$ is a solution to the enthalpy equation on $[0,R)$ with $\rho(0)=\rho_0$. Then for $r\in[0,R)$ we have
\[\rho(r)\leq\begin{cases}\displaystyle\brac{\rho_0^{-(2-\gamma)}+{2\pi\over d}{2-\gamma\over\gamma}r^2}^{-{1\over 2-\gamma}}&\qquad\text{when}\qquad\gamma\not=2\\\displaystyle\exp\brac{\ln(\rho_0)-{2\pi\over d}{1\over\gamma}r^2}&\qquad\text{when}\qquad\gamma=2\end{cases}.\]
Equivalently, we have when $\gamma>1$
\[w(r)\leq\begin{cases}\displaystyle\brac{w_0^{-(\alpha-1)}+{2\pi\over d}{2-\gamma\over\gamma}r^2}^{-{1\over\alpha-1}}&\qquad\text{when}\qquad\alpha\not=1\\\displaystyle\exp\brac{\ln(w_0)-{2\pi\over d}{\gamma-1\over\gamma}r^2}&\qquad\text{when}\qquad\alpha=1\end{cases}\]
and when $\gamma=1$,
\[h(r)\leq-\ln\brac{e^{-h_0}+{2\pi\over d}r^2}.\]
\end{lemma}
\begin{proof}

\mybb{Case 1: $\gamma>1$}

From (\ref{*}) we have
\begin{align*}
w'(r)&=-4\pi{\gamma-1\over\gamma}{1\over r^{d-1}}\int_0^ry^{d-1}w(y)^\alpha dy\\
&\leq-4\pi{\gamma-1\over\gamma}{1\over r^{d-1}}w(r)^\alpha\int_0^ry^{d-1}dy=-{4\pi\over d}{\gamma-1\over\gamma}w(r)^\alpha r
\end{align*}
where we have the first inequality because $w$ is decreasing. Rearranging and integrating we get
\begin{align*}
{2\pi\over d}{\gamma-1\over\gamma}r^2&\leq-\int_0^r{w'(y)\over w(y)^\alpha}dy=\int_{w(r)}^{w_0}{1\over z^\alpha}dz\\
&=\begin{cases}\displaystyle{1\over\alpha-1}\brac{{1\over w(r)^{\alpha-1}}-{1\over w_0^{\alpha-1}}}&\qquad\text{when}\qquad\alpha\not=1\\\ln w_0-\ln w(r)&\qquad\text{when}\qquad\alpha=1\end{cases}.
\end{align*}
Rearranging gives the desired result.

\mybb{Case 2: $\gamma=1$}

From (\ref{**}) we have
\begin{align*}
h'(r)&=-4\pi{1\over r^{d-1}}\int_0^ry^{d-1}e^{h(y)}dy
\leq-4\pi{1\over r^{d-1}}e^{h(r)}\int_0^ry^{d-1}dy=-{4\pi\over d}e^{h(r)}r
\end{align*}
where we have the first inequality because $w$ is decreasing. Rearranging and integrating we get
\begin{align*}
{2\pi\over d}r^2&\leq-\int_0^rh'(y)e^{-h(y)}dy=e^{-h(r)}-e^{-h_0}.
\end{align*}
So
\[-h(r)\geq\ln\brac{e^{-h_0}+{2\pi\over d}r^2}.\qed\]
\end{proof}



In particular if $\rho_0>1$, then $\rho$ would reach 1 at some $r=r^*<\infty$ with $w'(r^*)<0$ (by (\ref{*})) or $h'(r^*)<0$ (by (\ref{**})). We can cut off the solution at this point to obtain a state state for fluid stars.


Note that if we consider gas stars, i.e. $p=\rho^\gamma$, then whether $R=\infty$ or not correspond to whether the star is compactly supported or not.

Following is the so-called Pohozaev integral for the Lane–Emden stars, more generally considered in \cite{Heinzle}.

\begin{proposition}[Pohozaev integral]
Suppose $\rho$ is a solution to the steady state equation on $[0,R)$. Then for $r\in[0,R)$ we have
\begin{align*}
&2\pi\brac{2-{2-\gamma\over\gamma}d}\int_0^r\rho(y)^\gamma y^{d-1}dy\\
&={1\over 2}\gamma(\gamma-1)\brac{\rho(r)^{-(2-\gamma)}\rho'(r)}^2r^d+4\pi{\gamma-1\over\gamma}\rho(r)^{\gamma}r^d+{1\over 2}(d-2)\gamma\rho(r)^{2\gamma-3}\rho'(r)r^{d-1}.
\end{align*}
Equivalently, we have for $\gamma>1$,
\begin{align*}
&2\pi{\gamma-1\over\gamma}\brac{{2d\over 1+\alpha}-(d-2)}\int_0^rw(y)^{\alpha+1}y^{d-1}dy\\
&={1\over 2}w'(r)^2r^d+4\pi\brac{\gamma-1\over\gamma}^2w(r)^{\alpha+1}r^d+{1\over 2}(d-2)w'(r)w(r)r^{d-1}.
\end{align*}
and for $\gamma=1$,
\[-4\pi\int_0^re^{h(y)}y^{d-1}dy=h'(r)r^{d-1}.\]
\end{proposition}
\begin{proof}

\mybb{Case 1: $\gamma>1$}

Recall $w$ satisfies the enthalpy equation
\[-4\pi{\gamma-1\over\gamma}w(r)^\alpha={1\over r^{d-1}}{d\over dr}\brac{r^{d-1}{dw\over dr}}=w''(r)+{d-1\over r}w'(r)\qquad\text{for}\qquad r\in[0,R).\]
Times the enthalpy equation by $w$ and integrating w.r.t. $r^{d-1}dr$ we have
\begin{align}
-4\pi{\gamma-1\over\gamma}\int_0^rw^{\alpha+1}y^{d-1}dy
&=\int_0^rw''wy^{d-1}dy+(d-1)\int_0^rw'wy^{d-2}dy\nonumber\\
&=w'(r)w(r)r^{d-1}-\int_0^r(w')^2y^{d-1}dy\label{*_1}
\end{align}
Times the enthalpy equation by $rw'$ and integrating w.r.t. $r^{d-1}dr$ we have
\begin{align}
-{4\pi\over 1+\alpha}{\gamma-1\over\gamma}\brac{w(r)^{\alpha+1}r^d-d\int_0^rw^{\alpha+1}y^{d-1}dy}
&=-4\pi{\gamma-1\over\gamma}\int_0^rw^{\alpha}w'y^{d}dy\nonumber\\
&=\int_0^rw''w'y^{d}dy+(d-1)\int_0^r(w')^2y^{d-1}dy\nonumber\\
&={1\over 2}w'(r)^2r^d+(d/2-1)\int_0^r(w')^2y^{d-1}dy\label{*_2}
\end{align}
where we used
\[\int_0^rw^{\alpha}w'y^{d}dy=w(r)^{\alpha+1}r^d-\int_0^r(\alpha w^{\alpha}w'y^d+dw^{\alpha+1}y^{d-1})dy.\]
Using (\ref{*_1}) to eliminate the $\int(w')^2y^{d-1}dy$ term in (\ref{*_2}) we get the Pohozaev integral.

\mybb{Case 2: $\gamma=1$}

Integrating the $h$-equation $-4\pi e^h=h''+(d-1)r^{-1}h'$ w.r.t. $r^{d-1}dr$ we have
\[-4\pi\int_0^re^{h(y)}y^{d-1}dy=\int_0^rh''(y)y^{d-1}dy+(d-1)\int_0^rh'(y)y^{d-2}dy=h'(r)r^{d-1}.\qed\]
\end{proof}

\begin{theorem}[Support of Lane–Emden stars]\leavevmode
\begin{enumerate}
\item Suppose $w$ is a gas star. Then $w$ has compact support if
\[\gamma>{2d\over d+2}\qquad\text{ or equivalently }\qquad\alpha<{d+2\over d-2}\]
and infinitely support otherwise.
\item (Explicit solution when $\gamma={2d\over d+2}$) When $\gamma=2d/(d+2)$ we have explicit steady state solution
\[w(r)=A\brac{1+{2\pi\over d^2}A^{4\over d-2}r^2}^{1-d/2}\quad\text{or equivalently}\quad\rho(r)=C\brac{1+{2\pi\over d^2}C^{4\over d+2}r^2}^{-1-d/2}.\]
And the support of the liquid star is
\begin{align*}
R=\brac{{d^2\over 2\pi}C^{-{4\over d+2}}(C^{2\over d+2}-1)}^{1\over 2}.
\end{align*}
\end{enumerate}
\end{theorem}
\begin{proof}
Both of these are known standard results (see for example \cite{Heinzle} for i.). i. can be proven from the Pohozaev integral. We included the proof in the appendix for completeness.\qed
\end{proof}

\begin{proposition}[Self-similarity of solutions]
Let $\rho$ be a gaseous steady state. Then $\rho_\kappa(r)=\kappa\rho(\kappa^{1-\gamma/2}r)$ is a gaseous steady state for any $\kappa>0$, and the corresponding liquid star has support $R=\kappa^{-(1-\gamma/2)}\rho^{-1}(1/\kappa)$.
\end{proposition}
\begin{proof}
This is standard result based on scaling argument, we included the proof in the appendix for completeness.\qed
\end{proof}

\subsection{Singular solutions to the steady state equation}

\begin{proposition}[Singular star]
When $2(d-1)-d\gamma\geq 0$, the following solve the steady state equation on $(0,\infty)$
\[\rho(r)=\brac{{1\over 2\pi}{-d\gamma^2+2(d-1)\gamma\over(2-\gamma)^2}}^{1\over 2-\gamma}r^{-{2\over 2-\gamma}}.\]
And this is the only solution of the form $\rho(r)=Ar^a$.
\end{proposition}
\begin{proof}
First we consider the $\gamma>1$ case. Consider $w(r)=Ar^a$. Substitute into the enthalpy equation
\[w''+(d-1)r^{-1}w'=-4\pi{\gamma-1\over\gamma}w^\alpha\]
we get
\begin{align*}
Aa(a-1)r^{a-2}+A(d-1)ar^{a-2}=-4\pi{\gamma-1\over\gamma}A^\alpha r^{a\alpha}
\end{align*}
For this to have the possibility of holding, we need
\[a\alpha=a-2\qquad\iff\qquad a=-{2\over\alpha-1}=-{1\over\beta}=-2{\gamma-1\over 2-\gamma}.\]
Then the equation becomes
\begin{align*}
-2{\gamma-1\over 2-\gamma}{-\gamma\over 2-\gamma}+-2(d-1){\gamma-1\over 2-\gamma}&=-4\pi{\gamma-1\over\gamma}A^{\alpha-1}\\
\iff\qquad -\gamma^2+(d-1)(2-\gamma)\gamma&=2\pi(2-\gamma)^2A^{2\beta}
\end{align*}
So we need
\begin{align*}
A=\brac{{1\over 2\pi}{-d\gamma^2+2(d-1)\gamma\over(2-\gamma)^2}}^{1\over 2\beta}
\end{align*}
We also need
\begin{align*}
-d\gamma^2+2(d-1)\gamma\geq 0\qquad\iff\qquad 2(d-1)-d\gamma\geq 0.
\end{align*}

For the $\gamma=1$ case, substituting $\rho(r)=Ar^a$ in the steady state equation
\[0=\lpc(\ln\rho)+4\pi\rho={\rho''\over\rho}-{(\rho')^2\over\rho^2}+(d-1){1\over r}{\rho'\over \rho}+4\pi\rho\]
we get
\[0=(a(a-1)-a^2+(d-1)a)r^{-2}+4\pi Ar^a.\]
So we need $a=-2$ and $6-4-2(d-1)+4\pi A=0$.\qed
\end{proof}

\subsection{Dynamical system formulation of the steady state equation}

In order to prove our main (in)stability theorem for liquid stars, we will need a precise estimate for the radius of the liquid star, and due to the self-similar scaling of steady state solutions, this means we need to understand the precise tail behaviour of gaseous stars. And to do that we reformulate the Lane–Emden stars as a solutions to a dynamical system, and utilise methods of dynamical system analogous to that in \cite{Hadzic Lin Rein}.

Let
\[m(r)=4\pi\int_0^ry^{d-1}\rho(y)dy.\]
\begin{alignat*}{3}
\intertext{When $\gamma>1$, the steady state equation (\ref{*}) is}
{dw\over dr}&=-{\gamma-1\over\gamma}{m(r)\over r^{d-1}}&\qquad\text{where}\qquad w&=\rho^{\gamma-1}\\
\intertext{When $\gamma=1$, the steady state equation (\ref{**}) is}
{dh\over dr}&=-{m(r)\over r^{d-1}}&\qquad\text{where}\qquad h&=\ln\rho.
\end{alignat*}
Let
\begin{align*}
u_1(r)&=r^{{2\over 2-\gamma}}\rho(r)\\
u_2(r)&=r^{{2\over 2-\gamma}-d}m(r)
\end{align*}
The steady state equation for $\gamma>1$ is then
\begin{alignat*}{3}
{du_1\over dr}&=r^{{2\over 2-\gamma}}\alpha w^{\alpha-1}{dw\over dr}+{2\over 2-\gamma}r^{{2\over 2-\gamma}-1}\rho
&&=-{1\over\gamma}r^{-1}u_1^{2-\gamma}u_2+{2\over 2-\gamma}r^{-1}u_1\\
{du_2\over dr}&=4\pi r^{{2\over 2-\gamma}-1}\rho-\brac{d-{2\over 2-\gamma}}r^{{2\over 2-\gamma}-d-1}m
&&=4\pi r^{-1}u_1-\brac{d-{2\over 2-\gamma}}r^{-1}u_2
\end{alignat*}
And the steady state equation for $\gamma=1$ is
\begin{alignat*}{3}
{du_1\over dr}&=-r^{3-d}\rho m+2r\rho
&&=-r^{-1}u_1u_2+2r^{-1}u_1\\
{du_2\over dr}&=4\pi r\rho-(d-2)r^{1-d}m
&&=4\pi r^{-1}u_1-(d-2)r^{-1}u_2
\end{alignat*}
Now let $v_j(\tau)=u_j(e^\tau)$, i.e. the change of variable $\tau=\ln r$, then we obtain the planar autonomous dynamical system
\[\begin{array}{r!=l}
\displaystyle{dv_1\over d\tau}&\displaystyle-{1\over\gamma}v_1^{2-\gamma}v_2+{2\over 2-\gamma}v_1\\
\displaystyle{dv_2\over d\tau}&\displaystyle 4\pi v_1-\brac{d-{2\over 2-\gamma}}v_2
\end{array}\]
or equivalently
\[{d\mb v\over d\tau}=\mb F(\mb v)\qquad\text{where}\qquad\mb F(\mb v)=\mat{\displaystyle-{1\over\gamma}v_1^{2-\gamma}v_2+{2\over 2-\gamma}v_1\\\displaystyle 4\pi v_1-\brac{d-{2\over 2-\gamma}}v_2}.\]
Note that $\mb F\in C^\infty((0,\infty)\times\R)\cap C([0,\infty)\times\R)$.

The following two propositions established bounds for $\rho,u_1,u_2$, which will be needed later on when we apply results from dynamical systems to prove the tail behaviour of $\rho$.

\begin{proposition}\label{Buchdahl's inequality}
Suppose $\gamma<2$. Then
\begin{align*}
\rho(r)&\leq\brac{{2\pi\over d}{2-\gamma\over\gamma}}^{-{1\over 2-\gamma}}r^{-{2\over 2-\gamma}}\\
m(r)&\leq 4\pi\brac{d-{2\over 2-\gamma}}^{-1}\brac{{2\pi\over d}{2-\gamma\over\gamma}}^{-{1\over 2-\gamma}}r^{d-{2\over 2-\gamma}}
\end{align*}
\end{proposition}
\begin{proof}
The first inequality follows from the decay estimates. Then we have
\begin{align*}
m(r)&=4\pi\int_0^ry^{d-1}\rho(y)dy\\
&\leq 4\pi\brac{{2\pi\over d}{2-\gamma\over\gamma}}^{-{1\over 2-\gamma}}\int_0^ry^{d-1-{2\over 2-\gamma}}dy\\
&=4\pi\brac{d-{2\over 2-\gamma}}^{-1}\brac{{2\pi\over d}{2-\gamma\over\gamma}}^{-{1\over 2-\gamma}}r^{d-{2\over 2-\gamma}}.\qed
\end{align*}
\end{proof}

\begin{proposition}\label{u asymptotic}
We have
\begin{alignat*}{3}
u_1(r)&\sim r^{2\over 2-\gamma}\rho(0)&\qquad\text{as}\qquad r\to 0\\
u_2(r)&\sim r^{2\over 2-\gamma}\rho(0){4\pi\over d}&\qquad\text{as}\qquad r\to 0
\end{alignat*}
\end{proposition}
\begin{proof}
We have
\begin{align*}
{u_1(r)\over r^{2\over 2-\gamma}\rho(0)}={\rho(r)\over\rho(0)}\to 1\qquad\text{as}\qquad r\to 0
\end{align*}
\begin{align*}
{u_2(r)\over r^{2\over 2-\gamma}\rho(0){4\pi\over d}}&={d\over\rho(0)}r^{-d}\int_0^ry^{d-1}\rho(y)dy
={d\over\rho(0)}r^{-d}\int_0^ry^{d-1}(\rho(0)+o(1))dy\\
&={d\over\rho(0)}\brac{{\rho(0)\over d}+r^{-d}o(r^d)}
\to 1\qquad\text{as}\qquad r\to 0.\qed
\end{align*}
\end{proof}

\subsubsection{Steady states of the dynamical system}

In order to apply results from dynamical systems to our case, we need to know the steady states of the dynamical system. The following lemma detailed the steady states of the dynamical system and their property.

\begin{lemma}
Let $2(d-1)-d\gamma>0$. The dynamical system
\[{d\mb v\over d\tau}=\mb F(\mb v)\qquad\text{where}\qquad\mb F(\mb v)=\mat{\displaystyle-{1\over\gamma}v_1^{2-\gamma}v_2+{2\over 2-\gamma}v_1\\\displaystyle 4\pi v_1-\brac{d-{2\over 2-\gamma}}v_2}\]
has two steady states
\begin{align*}
\mb v&=\mb 0\qquad\text{and}\qquad\\
\mb v&=\mb v^*:=\mat{\displaystyle\brac{{1\over 2\pi}{\gamma\over 2-\gamma}\brac{d-{2\over 2-\gamma}}}^{1\over 2-\gamma}\\\displaystyle{2\gamma\over 2-\gamma}\brac{{1\over 2\pi}{\gamma\over 2-\gamma}\brac{d-{2\over 2-\gamma}}}^{\gamma-1\over 2-\gamma}}
=\mat{\displaystyle\brac{{1\over 2\pi}{-d\gamma^2+2(d-1)\gamma\over(2-\gamma)^2}}^{1\over 2-\gamma}\\\displaystyle{2\gamma\over 2-\gamma}\brac{{1\over 2\pi}{-d\gamma^2+2(d-1)\gamma\over(2-\gamma)^2}}^{\gamma-1\over 2-\gamma}}.
\end{align*}
Moreover, if $\gamma<{2d\over d+2}$, then $\mb v^*$ is (exponentially) stable.
\end{lemma}
\begin{proof}
It can be checked directly that $\mb F(\mb v)=\mb 0$ iff $\mb v=\mb 0$ or $\mb v=\mb v^*$. Note that $\mb F$ is differentiable at $\mb v^*$, but not at $\mb 0$ (unless $\gamma=1$). On $(0,\infty)\times\R$ we have
\begin{align*}
\grad\mb F(\mb v)=\mat{\displaystyle-{2-\gamma\over\gamma}v_1^{1-\gamma}v_2+{2\over 2-\gamma}&\displaystyle-{1\over\gamma}v_1^{2-\gamma}\\
\displaystyle 4\pi&\displaystyle-\brac{d-{2\over 2-\gamma}}}
\end{align*}
So
\begin{align*}
\grad\mb F(\mb v^*)&=\mat{\displaystyle{2\over 2-\gamma}-2&\displaystyle-{1\over 2\pi}{-d\gamma+2(d-1)\over(2-\gamma)^2}\\4\pi&\displaystyle-\brac{d-{2\over 2-\gamma}}}
\end{align*}

The eigenvalues of $\grad\mb F(\mb v^*)$ are
\[\lambda={2\over 2-\gamma}-1-{d\over 2}\pm{1\over 2}\sqrt{(d-2)^2-8{-d\gamma+2(d-1)\over(2-\gamma)^2}}.\]
So assuming
\[d+2>{4\over 2-\gamma}\qquad\iff\qquad\gamma<{2d\over d+2}\]
the system at $\mb v^*$ is (exponentially) stable provided
\begin{align*}
0&>\Re\brac{{2\over 2-\gamma}-1-{d\over 2}+{1\over 2}\sqrt{(d-2)^2-8{-d\gamma+2(d-1)\over(2-\gamma)^2}}}
\end{align*}
\begin{alignat*}{3}
&\iff\qquad&\brac{d+2-{4\over 2-\gamma}}^2&>(d-2)^2-8{-d\gamma+2(d-1)\over(2-\gamma)^2}\\
&\iff\qquad&8d&>{8\over(2-\gamma)^2}\brac{d\gamma-2(d-1)-2}\\
&\iff\qquad&d(2-\gamma)^2&>d\gamma-2(d-1)-2\\
&\iff\qquad&d(2-\gamma)((2-\gamma)+1)&>0\\
&\iff\qquad&d(2-\gamma)(3-\gamma)&>0
\end{alignat*}
which is true.\qed
\end{proof}

\subsection{Tail behaviour for gaseous star}

The following result gives detailed estimate for the tail behaviour for gaseous stars, and hence the boundary behaviour of liquid stars, that is crucial to prove our main (in)stability result in the next section.

\begin{theorem}
Suppose $\gamma<2d/(d+2)$. Then the gas star tends asymptotically to the singular star as $r\to\infty$. More precisely, there exist $c>0$ such that
\begin{align*}
\abs{r^{{2\over 2-\gamma}}\rho(r)-\brac{{1\over 2\pi}{-d\gamma^2+2(d-1)\gamma\over(2-\gamma)^2}}^{1\over 2-\gamma}}=\abs{u_1(r)-v^*_1}\lesssim r^{-c}\\
\abs{r^{{2\over 2-\gamma}-d}m(r)-{2\gamma\over 2-\gamma}\brac{{1\over 2\pi}{-d\gamma^2+2(d-1)\gamma\over(2-\gamma)^2}}^{\gamma-1\over 2-\gamma}}=\abs{u_2(r)-v^*_2}\lesssim r^{-c}
\end{align*}
\end{theorem}
\begin{proof}
We will work with the dynamical system formulation formulated in the previous section. From proposition \ref{Buchdahl's inequality} and the non-negativity of $\rho$ and $m$ we have
\begin{align*}
0&<v_1\leq\brac{{2\pi\over d}{2-\gamma\over\gamma}}^{-{1\over 2-\gamma}}\\
0&<v_2\leq 4\pi\brac{d-{2\over 2-\gamma}}^{-1}\brac{{2\pi\over d}{2-\gamma\over\gamma}}^{-{1\over 2-\gamma}}
\end{align*}
We will show that there exist $\epsilon'>0$ and $T\in\R$ such that $v_1(\tau)\geq\epsilon'$ for all $\tau\geq T$. By proposition \ref{u asymptotic} we have $v_1(\tau)\sim e^{{2\tau/(2-\gamma)}}\rho(0)$ as $\tau\to-\infty$. So for $\epsilon>0$ small enough, we can find $\tau_1\in\R$ such that $v_1(\tau_1)=\epsilon$ and $v_1(\tau_1+\delta)>\epsilon$ for small $\delta>0$. If $v_1(\tau)>\epsilon$ for all $\tau>\tau_1$ then we are done. Otherwise there exist a least $\tau_2>\tau_1$ such that $v_1(\tau_2)=\epsilon$. Since $F_1(\epsilon,v_2(\tau_1))>0$, from the expression of $F_1$ we see that $F_1(\epsilon,y)>0$ for all $y\in[0,v_2(\tau_1)]$. So we must have $v_2(\tau_2)>v_2(\tau_1)$. We must have $F_1(\epsilon,v_2(\tau_2))\leq 0$. We claim that there exist $\tau_3\geq\tau_2$ such that $F_1(\mb v(\tau_3))=0$. Suppose no such point exist, then $F_1(\mb v(\tau))<0$ for all $\tau\geq\tau_2$, in other words
\begin{align}
{1\over\gamma}v_1(\tau)^{2-\gamma}v_2(\tau)>{2\over 2-\gamma}v_1(\tau)\qquad\text{for all}\qquad\tau\geq\tau_2\label{star}.
\end{align}
We will show that this is impossible. Using (\ref{star}) we have for all $\tau\geq\tau_2$
\[F_2(\mb v(\tau))=4\pi v_1(\tau)-\brac{d-{2\over 2-\gamma}}v_2(\tau)<4\pi\epsilon-\brac{d-{2\over 2-\gamma}}\epsilon^{\gamma-1}{2\gamma\over 2-\gamma}.\]
By choosing $\epsilon$ small enough, we can make $F_2(\mb v(\tau))$ less than a fix strictly negative number for all $\tau\geq\tau_2$. This means $v_2(\tau)\to 0$ as $\tau\to\infty$.
\begin{enumerate}
\item When $\gamma=1$, (\ref{star}) gives $v_2(\tau)>2$ for all $\tau\geq\tau_2$. This is a contradiction.
\item When $\gamma>1$, (\ref{star}) gives 
\[\brac{{2-\gamma\over 2\gamma}v_2(\tau)}^{1\over\gamma-1}>v_1(\tau)\qquad\text{for all}\qquad\tau\geq\tau_2.\]
So for all $\tau\geq\tau_2$ we have
\begin{align*}
{dv_2\over d\tau}&=4\pi v_1-\brac{d-{2\over 2-\gamma}}v_2\\
&<4\pi\brac{{2-\gamma\over 2\gamma}v_2}^{1\over\gamma-1}-\brac{d-{2\over 2-\gamma}}v_2\\
&\sim-\brac{d-{2\over 2-\gamma}}v_2\qquad\text{as}\qquad v_2\to 0.
\end{align*}
So we must have
\begin{alignat*}{3}
|v_1(\tau)|&\to 0\qquad &&\text{as}\qquad\tau\to\infty\\
|v_2(\tau)|&=O(e^{-(d-{2\over 2-\gamma})\tau})\qquad &&\text{as}\qquad\tau\to\infty.
\end{alignat*}
In other words we have
\begin{alignat*}{3}
|u_1(r)|&\to 0\qquad &&\text{as}\qquad r\to\infty\\
|u_2(r)|&=O(r^{-(d-{2\over 2-\gamma})})\qquad &&\text{as}\qquad r\to\infty.
\end{alignat*}
The Pohozaev integral gives
\begin{align*}
2\pi{\gamma-1\over\gamma}\brac{2-{2-\gamma\over\gamma}d}\int_0^r\rho(y)^\gamma y^{d-1}dy
&\geq{1\over 2}(d-2)w'(r)w(r)r^{d-1}\\
&=-{1\over 2}{\gamma-1\over\gamma}(d-2)u_1(r)^{\gamma-1}u_2(r)r^{d-{2\over 2-\gamma}-2{\gamma-1\over 2-\gamma}}\\
&\to 0\qquad\text{as}\qquad r\to\infty
\end{align*}
but the LHS of the equation becomes more and more negative as $r\to\infty$. This is a contradiction
\end{enumerate}
Therefore we have $\tau_3\geq\tau_2$ such that $F_1(\mb v(\tau_3))=0$ and $F_2(\mb v(\tau_3))<0$. From the expression for $\mb F$ we see that $F_1(v_1(\tau_3),x)\geq 0$ for all $x\in[0,v_2(\tau_3)]$. Let $\tau_0\leq\tau_1$ be the point such that $v_1(\tau_0)=v_1(\tau_3)$. Since $F_1(\mb v(\tau_0))>0$ and $F_1(\mb v(\tau_3))=0$, we see from the expression of $\mb F$ that we must have $v_2(\tau_3)>v_2(\tau_0)$. Hence for $\tau>\tau_3$, $v(\tau)$ stays in the region bounded by the arc $v([v_0,v_3])$ and the line $\{(v_1(\tau_3),x):x\in[v_2(\tau_0),v_2(\tau_3)]\}$. Hence $T=\tau_3$ works.

So $\{\mb v(\tau):\tau\geq T\}$ lies in a compact set within the region $v_1>\epsilon/2$ where $\mb F$ is $C^1$. This means its $\omega$-limit set is non-empty, compact and connected by a standard result in dynamical systems. By the Poincar\'e-Bendixson theorem, the $\omega$-limit set must be either
\begin{enumerate}[(a)]
\item $\{\mb v^*\}$;
\item a periodic orbit;
\item homoclinic obits connecting $\mb v^*$.
\end{enumerate}

But by Bendixson-Dulac theorem, there is no periodic orbits in $(0,\infty)\times(0,\infty)$ since on this region
\begin{align*}
\grad\cdot\brac{{1\over v_1^{2-\gamma}}\mb F(\mb v)}&=2{\gamma-1\over 2-\gamma}v_1^{-(2-\gamma)}-\brac{d-{2\over 2-\gamma}}v_1^{-(2-\gamma)}
=-\brac{d-{2\gamma\over 2-\gamma}}v_1^{-(2-\gamma)}<0
\end{align*}
where we note that
\[d>{2\gamma\over 2-\gamma}\qquad\iff\qquad\gamma<{2d\over d+2}.\]
Since $\mb v^*$ is (exponentially) stable, the $\omega$-limit set cannot have homoclinic orbits connecting $\mb v^*$ either. So the $\omega$-limit set of $\mb v$ must be $\{\mb v^*\}$. So $\mb v(\tau)\to\mb v^*$ as $\tau\to\infty$. Since the fixed point $\mb v^*$ is exponentially stable, we have $\|\mb v(\tau)-\mb v^*\|\lesssim e^{-c\tau}$ for some $c>0$. Converting to the variable $r$ gives us the desired result.\qed
\end{proof}

\begin{corollary}\label{Large central density limit}
Suppose $\gamma<2d/(d+2)$. Let $\rho$ be a gaseous steady state and $\rho_\kappa(r)=\kappa\rho(\kappa^{1-\gamma/2}r)$. Then there exist $c>0$ such that
\begin{alignat*}{3}
\abs{r^{{2\over 2-\gamma}}\rho_\kappa(r)-\brac{{1\over 2\pi}{-d\gamma^2+2(d-1)\gamma\over(2-\gamma)^2}}^{1\over 2-\gamma}}&=\abs{r^{{2\over 2-\gamma}}\rho_\kappa(r)-v^*_1}&\lesssim(\kappa^{1-\gamma/2}r)^{-c}\\
\abs{r^{{2\over 2-\gamma}-d}m_\kappa(r)-{2\gamma\over 2-\gamma}\brac{{1\over 2\pi}{-d\gamma^2+2(d-1)\gamma\over(2-\gamma)^2}}^{\gamma-1\over 2-\gamma}}&=\abs{r^{{2\over 2-\gamma}-d}m_\kappa(r)-v^*_2}&\lesssim(\kappa^{1-\gamma/2}r)^{-c}
\end{alignat*}
\end{corollary}
\begin{proof}
Using the last theorem we have
\begin{align*}
\abs{r^{{2\over 2-\gamma}}\rho_\kappa(r)-v^*_1}=\abs{(\kappa^{1-\gamma/2}r)^{{2\over 2-\gamma}}\rho(\kappa^{1-\gamma/2}r)-v^*_1}\lesssim(\kappa^{1-\gamma/2}r)^{-c}.
\end{align*}
We have
\begin{align*}
m_\kappa(r)&=4\pi\int_0^ry^{d-1}\rho_\kappa(y)dy=4\pi\kappa^{1-(d-1)(1-\gamma/2)}\int_0^r(\kappa^{1-\gamma/2}y)^{d-1}\rho(\kappa^{1-\gamma/2}y)dy\\
&=4\pi\kappa^{1-d(1-\gamma/2)}\int_0^{\kappa^{1-\gamma/2}r}z^{d-1}\rho(z)dz
=\kappa^{1-d(1-\gamma/2)}m(\kappa^{1-\gamma/2}r)
\end{align*}
So using the last theorem we have
\begin{align*}
\abs{r^{{2\over 2-\gamma}-d}m_\kappa(r)-v^*_2}=\abs{(\kappa^{1-\gamma/2}r)^{{2\over 2-\gamma}-d}m(\kappa^{1-\gamma/2}r)-v^*_2}\lesssim(\kappa^{1-\gamma/2}r)^{-c}.\qed
\end{align*}
\end{proof}

\begin{corollary}\label{radius limit}
Suppose $\gamma<2d/(d+2)$. Let $\rho$ be a gaseous steady state and $\rho_\kappa(r)=\kappa\rho(\kappa^{1-\gamma/2}r)$. Then the liquid star $\rho_\kappa$ has radius
\begin{align*}
R_\kappa\to R_\infty:=(v^*_1)^{1-\gamma/2}
=\brac{{1\over 2\pi}{-d\gamma^2+2(d-1)\gamma\over(2-\gamma)^2}}^{1\over 2}
\qquad\text{as}\qquad\kappa\to\infty.
\end{align*}
\end{corollary}
\begin{proof}
From the last theorem, we have
\begin{align*}
\abs{\rho^{-1}\brac{1\over\kappa}^{2\over 2-\gamma}{1\over\kappa}-v^*_1}\lesssim\rho^{-1}\brac{1\over\kappa}^{-c}\to 0\qquad\text{as}\qquad\kappa\to\infty.
\end{align*}
So we have
\begin{align*}
R_\kappa&=\kappa^{-(1-\gamma/2)}\rho^{-1}(1/\kappa)=\brac{\rho^{-1}\brac{1\over\kappa}^{2\over 2-\gamma}{1\over\kappa}}^{1-\gamma/2}
\to(v^*_1)^{1-\gamma/2}\qquad\text{as}\qquad\kappa\to\infty.\qed
\end{align*}
\end{proof}

\section{Linear stability for radially symmetric perturbations}

\subsection{Equations with spherical symmetry}

Assuming spherical symmetry, so that $\mb u(r,t)=u(r,t)\hat{\mb r}$, the continuity equation $\partial_t\rho+\grad\cdot(\rho\mb u)=0$ becomes
\begin{align}
0&=\partial_t\rho+\grad\cdot(\rho\mb u)=\partial_t\rho+\rho u\grad\cdot\hat{\mb r}+\hat{\mb r}\cdot\grad(\rho u)=\partial_t\rho+(d-1)\rho u{1\over r}+\hat{\mb r}\cdot\partial_r(\rho u)\grad r\nonumber\\
&=\partial_t\rho+(d-1)\rho u{1\over r}+\partial_r(\rho u)=D_t\rho+\rho\brac{(d-1)u{1\over r}+\partial_ru}
=D_t\rho+\rho{\partial_r(r^{d-1}u)\over r^{d-1}}\label{continuity equation 3}.
\end{align}
The momentum equation (\ref{momentum equation 2}) reads
\begin{align*}
\mb 0&=\rho{D\mb u\over Dt}+\grad p+\rho\grad\phi
=\rho(\partial_tu\hat{\mb r}+(u\hat{\mb r}\cdot\grad)(u\hat{\mb r}))+(\partial_rp)\hat{\mb r}+\rho(\partial_r\phi)\hat{\mb r}\\
&=\rho(\partial_tu\hat{\mb r}+u\partial_r(u\hat{\mb r}))+(\partial_rp)\hat{\mb r}+\rho(\partial_r\phi)\hat{\mb r}\\
&=\brac{\rho(\partial_tu+u\partial_ru+\partial_r\phi)+\partial_rp}\hat{\mb r}
\end{align*}
So the momentum equation is
\[D_tu+\partial_r\phi+{1\over\rho}\partial_rp=\partial_tu+u\partial_ru+\partial_r\phi+{1\over\rho}\partial_rp=0.\]
The Poisson equation reads
\[4\pi\rho=\lpc\phi=\grad\cdot((\partial_r\phi)\hat{\mb r})=(\partial_r\phi){d-1\over r}+\partial_r^2\phi={1\over r^{d-1}}\partial_r(r^{d-1}\partial_r\phi)).\]
We can put this into the momentum equation to get
\begin{align}
\partial_tu+u\partial_ru+{4\pi\over r^{d-1}}\int_0^rs^{d-1}\rho(s)ds+{1\over\rho}\partial_rp=0.\label{momentum equation 3}
\end{align}

\subsection{Equations in Lagrangian coordinates}

Let $\eta(y,t)$ be the (radial) location of the fluid particle that was at $\eta_0(y)$ at time 0. $\eta$ is given by
\[\partial_t\eta=u\circ\eta\qquad\text{with}\qquad\eta(y,0)=\eta_0(y)\]
where $u\circ\eta(y,t)=u(\eta(y,t),t)$. The spacial domain is then fixed for all time as $[0,R]:=\eta_0^{-1}(\{r:r\hat{\mb r}\in\Omega_0\})$. We then have the Lagrangian variables
\begin{align*}
v&=u\circ\eta\hfill\tag{Lagrangian velocity}\\
f&=\rho\circ\eta\hfill\tag{Lagrangian density}\\
\psi&=\phi\circ\eta\hfill\tag{Lagrangian potential}
\end{align*}
W have for any $h$,
\[\partial_y(h\circ\eta)=((\partial_rh)\circ\eta)\partial_y\eta\qquad\text{ and so }\qquad(\partial_rh)\circ\eta=(\partial_y\eta)^{-1}\partial_y(h\circ\eta).\]
\[\partial_t(h\circ\eta)=(\partial_th)\circ\eta+((\partial_rh)\circ\eta)\partial_t\eta=(D_th)\circ\eta.\]

Let
\[J={\eta^{d-1}\over y^{d-1}}\partial_y\eta={1\over dy^{d-1}}\partial_y(\eta^d).\]
Then
\begin{align*}
\partial_tJ&={1\over dy^{d-1}}\partial_y(d\eta^{d-1}\partial_t\eta)
={1\over y^{d-1}}\brac{(d-1)\eta^{d-2}v\partial_y\eta+\eta^{d-1}\partial_yv}\\
&={1\over y^{d-1}}\brac{(d-1)\eta^{d-2}v\partial_y\eta+\eta^{d-1}((\partial_ru)\circ\eta)\partial_y\eta}
=(d-1)J{v\over\eta}+J((\partial_ru)\circ\eta)\\
&=J\brac{\partial_r(r^{d-1}u)\over r^{d-1}}\circ\eta
\end{align*}
So the continuity equation (\ref{continuity equation 3}) in Lagrangian is
\[0=\partial_tf+f{\partial_tJ\over J}\qquad\text{ and so }\qquad\partial_t\ln(fJ)=0\qquad\text{ and so }\qquad fJ=f_0J_0.\]
The Poisson equation is
\[4\pi f={1\over\eta^{d-1}\partial_y\eta}\partial_y\brac{{\eta^{d-1}\over\partial_y\eta}\partial_y\psi}.\]
And so
\[{1\over\partial_y\eta}\partial_y\psi={4\pi\over d\eta^{d-1}}\int^y_0f(s)\partial_y\eta^d(s)ds={4\pi\over\eta^{d-1}}\int^y_0s^{d-1}(f_0J_0)(s)ds.\]
And the momentum equation (\ref{momentum equation 3}) is
\begin{align*}
0&=\partial_tv+{1\over\partial_y\eta}\partial_y\psi+{1\over f\partial_y\eta}\partial_yf^\gamma\\
&=\partial_t^2\eta+{4\pi\over\eta^{d-1}}\int^y_0s^{d-1}(f_0J_0)(s)ds+{1\over f_0J_0}{\eta^{d-1}\over y^{d-1}}\partial_y\brac{f_0J_0{dy^{d-1}\over\partial_y\eta^d}}^\gamma
\end{align*}

\subsection{The eigenvalue problem for the linear stability of steady states}

Let $\bar\rho$ be a steady state solution. For a small perturbation $\rho_0=\bar\rho+\epsilon$, we can pick $\eta_0$ such that $f_0J_0=\bar\rho$. And a small perturbation $u_0=\upsilon$ means $\partial_t\eta(y,0)=\upsilon\circ\eta_0(y,0)$. Then our equation becomes
\begin{align}
\partial_t^2\eta+{4\pi\over\eta^{d-1}}\int^y_0s^{d-1}\bar\rho(s)ds+{1\over\bar\rho}{\eta^{d-1}\over y^{d-1}}\partial_y\brac{\bar\rho{dy^{d-1}\over\partial_y\eta^d}}^\gamma=0.\label{momentum equation 4}
\end{align}
Since $\bar\rho$ is a steady state solution, we have
\[{4\pi\over y^{d-1}}\int^y_0s^{d-1}\bar\rho(s)ds+{1\over\bar\rho}\partial_y\bar\rho^\gamma=0.\]

To study linear stability, we linearised the equation about the steady state as follows.

\begin{proposition}[Linearised momentum equation for perturbation]
Let $\eta(y,t)=y$ correspond to the steady state solution $\rho=\bar\rho$. Consider a perturbation of this given by $\eta(y,t)=y(1+\zeta(y,t))$. Then the perturbation variable $\zeta(y,t)$ satisfies, to first order,
\begin{align*}
y\partial_t^2\zeta+{2(d-1)\over\bar\rho}\zeta\partial_y\bar\rho^\gamma-\gamma{1\over\bar\rho}\partial_y\brac{\bar\rho^{\gamma}(d\zeta+y\partial_y\zeta)}&=0\\
d\zeta(R,t)+R\partial_y\zeta(R,t)&=0.
\end{align*}
Furthermore, if $\zeta(y,t)=e^{\lambda t}\chi(y)$, then $\chi$ satisfies
\begin{align}
\underbrace{-\gamma\partial_y(\bar\rho^\gamma y^{d+1}\partial_y\chi)+(2(d-1)-d\gamma)y^d\chi\partial_y\bar\rho^\gamma}_{:=L\chi}&=-\lambda^2y^{d+1}\bar\rho\chi\label{SLP}\\
d\chi(R)+R\partial_y\chi(R)&=0\label{RBC}.
\end{align}
\end{proposition}
\begin{proof}
We have $\eta(y,t)=y(1+\zeta(y,t))$, so
\[\partial_y\eta=1+\zeta+y\partial_y\zeta.\]
Assuming $\zeta$ and $\partial_y\zeta$ is small, we have
\begin{align*}
\eta^{d-1}&=y^{d-1}(1+(d-1)\zeta+o(\zeta))\\
\eta^{-(d-1)}&=y^{-(d-1)}(1-(d-1)\zeta+o(\zeta))\\
\partial_y\eta^d&=d\eta^{d-1}\partial_y\eta=dy^{d-1}(1+(d-1)\zeta+o(\zeta))(1+\zeta+y\partial_y\zeta)\\
&=dy^{d-1}(1+d\zeta+y\partial_y\zeta+o(|\zeta|+|\partial_y\zeta|))\\
(\partial_y\eta^d)^{-\gamma}&=(dy^{d-1})^{-\gamma}\brac{1-\gamma(d\zeta+y\partial_y\zeta)+o(|\zeta|+|\partial_y\zeta|)}
\end{align*}
So the momentum equation (\ref{momentum equation 4}) is
\[y\partial_t^2\zeta+{4\pi\over y^{d-1}}(1-(d-1)\zeta)\int^y_0s^{d-1}\bar\rho(s)ds+{1\over\bar\rho}(1+(d-1)\zeta)\partial_y\brac{\bar\rho^\gamma(1-\gamma(d\zeta+y\partial_y\zeta))}+o(|\zeta|+|\partial_y\zeta|)=0.\]
Discarding non-linear terms and simplify we get the linearised momentum equation
\begin{align*}
0
&=y\partial_t^2\zeta+{2(d-1)\over\bar\rho}\zeta\partial_y\bar\rho^\gamma-\gamma{1\over\bar\rho}\partial_y\brac{\bar\rho^{\gamma}(d\zeta+y\partial_y\zeta)}.
\end{align*}
For solutions of the form $\zeta(y,t)=e^{\lambda t}\chi(y)$, we have that $\chi$ satisfies
\begin{align*}
0&=\lambda^2y\bar\rho\chi+2(d-1)\chi\partial_y\bar\rho^\gamma-\gamma\partial_y\brac{\rho^{\gamma}(d\chi+y\partial_y\chi)}\\
&=\lambda^2y\bar\rho\chi+(2(d-1)-d\gamma)\chi\partial_y\bar\rho^\gamma-\gamma y(\partial_y\chi)\partial_y\rho^\gamma-\gamma\rho^\gamma\partial_y(d\chi+y\partial_y\chi)\\
&=\lambda^2y\bar\rho\chi+(2(d-1)-d\gamma)\chi\partial_y\bar\rho^\gamma-{\gamma\over y^d}\partial_y(\bar\rho^\gamma y^{d+1}\partial_y\chi)
\end{align*}
So we want to solve a Sturm-Liouville type equation
\[\lambda^2y^{d+1}\bar\rho\chi=\gamma\partial_y(\bar\rho^\gamma y^{d+1}\partial_y\chi)-(2(d-1)-d\gamma)y^d\chi\partial_y\bar\rho^\gamma=:-L\chi\]
Since $fJ=f_0J_0=\bar\rho$ and $f(R)=\bar\rho(R)=1$, we have $J(R)=1$. Now
\[J={\eta^{d-1}\over y^{d-1}}\partial_y\eta=(1+\zeta)^{d-1}\partial_y(y(1+\zeta))=(1+\zeta)^{d-1}(1+\zeta+y\partial_y\zeta)=1+d\zeta+y\partial_y\zeta+o(|\zeta|+|\partial_y\zeta|).\]
Discarding the non-linear terms and evaluating at $R$ we get a Robin type boundary condition $d\zeta(R)+R\partial_y\zeta(R)=0$. In terms of $\chi$ this condition reads
\[d\chi(R)+R\partial_y\chi(R)=0.\qed\]
\end{proof}

In this paper, we say the system is linearly unstable to mean that the linearised equation admits an growing mode solution of the form $\zeta(y,t)=e^{\lambda t}\chi(y)$ with $\lambda>0$. Otherwise we call the system linearly stable.

Given $\chi_1,\chi_2\in C^2([0,R])$ satisfying the boundary condition (\ref{RBC}), under the usual $L^2$ inner product, we have using integration by parts
\begin{align*}
\<L\chi_1,\chi_2\>&=\int_0^R\gamma\bar\rho^\gamma y^{d+1}(\partial_y\chi_1)(\partial_y\chi_2)+(2(d-1)-d\gamma)y^d\chi_1\chi_2\partial_y\bar\rho^\gamma dy\\
&\quad-\gamma\bar\rho(R)^\gamma R^{d+1}\chi_2(R)\partial_y\chi_1(R)\\
&=-\int_0^R\chi_1\gamma\partial_y(\bar\rho^\gamma y^{d+1}\partial_y\chi_2)-(2(d-1)-d\gamma)y^d\chi_1\chi_2\partial_y\bar\rho^\gamma dy\\
&\quad-\gamma\bar\rho(R)^\gamma R^{d+1}(\chi_2(R)\partial_y\chi_1(R)-\chi_1(R)\partial_y\chi_2(R))\\
&=\<\chi_1,L\chi_2\>
\end{align*}
So $L$ is symmetric under $C^2([0,R])$ functions satisfying the (\ref{RBC}). Note that in particular
\begin{align}
\<L\chi,\chi\>=\int_0^R\gamma\bar\rho^\gamma y^{d+1}(\partial_y\chi)^2+(2(d-1)-d\gamma)y^d\chi^2\partial_y\bar\rho^\gamma dy+d\gamma R^{d}\chi(R)^2.\label{TE}
\end{align}
Define the bilinear form
\begin{align*}
Q[\chi_1,\chi_2]=\int_0^R\gamma\bar\rho^\gamma y^{d+1}(\partial_y\chi_1)(\partial_y\chi_2)+(2(d-1)-d\gamma)y^d\chi_1\chi_2\partial_y\bar\rho^\gamma dy+d\gamma R^{d}\chi_1(R)\chi_2(R).
\end{align*}
Note that this equation is well defined for spherically symmetric functions in $H^1(B_{R}(\R^{d+2}))$, where we consider $y$ the radial variable, because 1) the trace theorem for Sobolev space which meant that $\chi(R)$ is well defined; and 2) the fact that $\partial_y\bar\rho^\gamma\sim y$ which means the the integral was weighted by $\sim y^{d+1}$.

We want to solve $L\chi=-\lambda^2y^{d+1}\bar\rho\chi$ on $[0,R]$ with the boundary condition $d\chi(R)+R\partial_y\chi(R)=0$. If there exist a negative eigenvalue $\mu=-\lambda^2<0$, the we found a growing mode of the original linearised problem, so the system is unstable. So we want to find the smallest eigenvalue $\mu$. The following lemma help to make a criterion for finding $\mu$ in terms of the quadratic form $Q$.

\begin{lemma}
Let $H^1_r(B_{R}(\R^{d+2}))$ denote the subspace of spherically symmetric functions in $H^1(B_{R}(\R^{d+2}))$. We consider functions in $H^1_r(B_{R}(\R^{d+2}))$ to be functions of one variable defined by radial distance $y\in[0,R]$. In this space we have
\[\inf_{\|\chi\|_{y^{d+1}\bar\rho}=1}\<L\chi,\chi\>_{L^2([0,R])}=:\mu_*=\inf\{\mu:\exists\chi\not=0\st L\chi=\mu y^{d+1}\bar\rho\chi\}\]
where $\|\chi\|_{w}^2=\<\chi,w\chi\>_{L^2([0,R])}$. Moreover, the infimum is attained by some $\chi_*\in H^1_r(\R^{d+2})$ which is an eigenfunction of $L$ with eigenvalue $\mu_*$.
\end{lemma}
\begin{proof}
It is clear that
\[\mu_*=\inf_{\|\chi\|_{y^{d+1}\bar\rho}=1}\<L\chi,\chi\>\leq\inf\{\mu:\exists\chi\not=0\st L\chi=\mu y^{d+1}\bar\rho\chi\}.\]
To prove equality, it suffice then to prove that $\mu_*$ is an eigenvalue of $L$. 

Pick $\chi_n$ with $\|\chi_n\|_{y^{d+1}\bar\rho}=1$ such that $\<L\chi_n,\chi_n\>\to\inf_{\|\chi\|_{y^{d+1}\bar\rho}=1}\<L\chi,\chi\>$. Since
\[\int_0^R y^d\chi_n^2\partial_y\bar\rho^\gamma dy\sim\|\chi_n\|_{y^{d+1}\bar\rho}=1,\]
and the first and last term in (\ref{TE}) are positive, $\inf_{\|\chi\|_{y^{d+1}\bar\rho}=1}\<L\chi,\chi\>$ is finite. Now
\begin{align*}
\|\partial_y\chi_n\|^2_{L^2(B_R(\R^{d+2}))}
&\lesssim\int_0^R\gamma\bar\rho^\gamma y^{d+1}(\partial_y\chi_n)^2dy+d\gamma R^{d}\chi_n(R)^2\\
&=\<L\chi_n,\chi_n\>-(2(d-1)-d\gamma)\int_0^Ry^d\chi^2\partial_y\bar\rho^\gamma dy\\
&\lesssim|\<L\chi_n,\chi_n\>|+\|\chi_n\|_{y^{d+1}\bar\rho}.
\end{align*}
And obviously
\[\|\chi_n\|_{L^2(B_R(\R^{d+2}))}\lesssim\|\chi_n\|_{y^{d+1}\bar\rho}.\]
Hence $\chi_n$ is bounded in $H^1(B_R(\R^{d+2}))$. Wlog, picking an appropriate subsequence, we can assume $\chi_n$ converge weakly to some $\chi_*$. By the Rellich-Kondrachov theorem, $\chi_n\to\chi$ in $L^2(B_R(\R^{d+2}))$. It follows that $\|\chi_*\|_{y^{d+1}\bar\rho}=1$ By the lower semi-continuity of weak convergence, we have $\liminf\|\chi_n\|_{H^1(B_R(\E^{d+2}))}\geq\|\chi_*\|$. Since $\|\chi_n\|_{L^2(B_R(\E^{d+2}))}\to\|\chi_*\|_{L^2(B_R(\E^{d+2}))}$, we must have
\[\liminf\|\partial_y\chi_n\|_{L^2(B_R(\E^{d+2}))}^2\geq\|\partial_y\chi_*\|_{L^2(B_R(\E^{d+2}))}^2.\]
Since $\|\ph\|_{y^{d+1}\bar\rho}$ is an (equivalent) norm for $L^2(B_R(\E^{d+2}))$, we have
\[\liminf\int_0^R\gamma\bar\rho^\gamma y^{d+1}(\partial_y\chi_n)^2dy\geq\int_0^R\gamma\bar\rho^\gamma y^{d+1}(\partial_y\chi_*)^2dy.\]
Since the trace operator $T$ is continuous and linear, we also have $T\chi_n\rightharpoonup T\chi_*$ and so by the lower semi-continuity of weak convergence, $\liminf\chi_n(R)^2\geq\chi_*(R)^2$. It follows that $\<L\chi_*,\chi_*\>\leq\inf_{\|\chi\|_{y^{d+1}\bar\rho}=1}\<L\chi,\chi\>$, and that means we must have equality and the infimum is attained.

For any $f$ we must have
\begin{align*}
0&={d\over d\epsilon}\eva{\brac{Q[\chi_*+\epsilon f,\chi_*+\epsilon f]\over\<\chi_*+\epsilon f,\chi_*+\epsilon f\>_{y^{d+1}\bar\rho}}}_{\epsilon=0}
={d\over d\epsilon}\eva{\brac{Q[\chi_*,\chi_*]+2\epsilon Q[\chi_*,f]+\epsilon^2Q[f,f]\over\<\chi_*,\chi_*\>_{y^{d+1}\bar\rho}+2\epsilon\<\chi_*,f\>_{y^{d+1}\bar\rho}+\epsilon^2\<f,f\>_{y^{d+1}\bar\rho}}}_{\epsilon=0}\\
&={2Q[\chi_*,f]\over\<\chi_*,\chi_*\>_{y^{d+1}\bar\rho}}-{2Q[\chi_*,\chi_*]\<\chi_*,f\>_{y^{d+1}\bar\rho}\over\<\chi_*,\chi_*\>_{y^{d+1}\bar\rho}^2}
\end{align*}
and so $Q[\chi_*,f]=\mu_*\<\chi_*,f\>_{y^{d+1}\bar\rho}$. Hence $\chi_*$ is a weak solution to $L\chi=\mu_*\chi$. By elliptic regularity, $\chi_*$ is smooth on $(0,R]$, and so the weak solution is in fact a classical solution. Therefore $\chi_*$ is in fact an eigenfunction of $L$ with eigenvalue $\mu_*$.\qed
\end{proof}

From this we get the linear stability criterion in terms of whether the quadratic form $Q$ is non-negative definite.

\begin{proposition}
If $\<L\chi,\chi\>\geq 0$ for all $\chi$, then the corresponding liquid Lane-Emden star is linearly stable under radial perturbations. And if there exist $\chi$ such that $\<L\chi,\chi\><0$, then it must be linearly unstable.
\end{proposition}
\begin{proof}
If there exist $\chi$ such that $\<L\chi,\chi\><0$, then by the previous lemma there exist $-\mu<0$ and $\chi_*$ such that $L\chi_*=-\mu y^{d+1}\bar\rho\chi_*$. This, by the last proposition, means the linearised momentum equation admits a solution of the form $\zeta(y,t)=e^{\sqrt\mu t}\chi_*(y)$. This grows exponentially in time, and hence the corresponding liquid Lane-Emden star is linearly unstable. Conversely, if $\<L\chi,\chi\>\geq 0$ for all $\chi$, then no such growing solutions exist and hence the corresponding liquid Lane-Emden star is linearly stable under radial perturbations.\qed
\end{proof}

\subsection{(In)stability results}

We will now prove our main theorem \ref{liquid Lane–Emden stars stability} on the (in)stability results for liquid Lane–Emden stars. This will be split into three theorems below, that together will established theorem \ref{liquid Lane–Emden stars stability}.

\begin{theorem}
The liquid Lane–Emden stars is, against radial perturbations, linearly stable when $\gamma\geq 2(d-1)/d$.
\end{theorem}
\begin{proof}
Since $\partial_y\bar\rho^\gamma<0$, if $2(d-1)-d\gamma\leq 0$, then it's clear from equation (\ref{TE}) that $\<L\chi,\chi\>>0$ for all $\chi\not=0$, hence the system must be stable.\qed
\end{proof}

When $\gamma<2(d-1)/d$, the proof of linear stability for stars of small relative central density require Poincar\'e-Hardy-type inequalities in the following two lemmas and proposition.

\begin{lemma}
Let $v\in C^1([a,b])$, then
\begin{align*}
\int_a^b|v(z)|^2dz\lesssim_{a,b}|v(a)|^2+\int_a^b|v'(z)|^2dz\\
\int_a^b|v(z)|^2dz\lesssim_{a,b}|v(b)|^2+\int_a^b|v'(z)|^2dz.
\end{align*}
\end{lemma}
\begin{proof}
We will prove the second statement, the first is proven in the same way. By the fundamental theorem of calculus
\[v(z)=v(b)-\int_z^bv'(y)dy.\]
Using the fact that $|x+y|^2\leq|x|^2+2|x||y|+|y|^2\leq 2|x|^2+2|y|^2$ and H\"older's inequality we have
\[|v(z)|^2\leq 2|v(b)|^2+2\abs{\int_z^bv'(y)dy}^2\leq 2|v(b)|^2+2(b-z)^2\int_z^b|v'(y)|^2dy.\]
So
\begin{align*}
\int_a^b|v(z)|^2dz&\leq\int_a^b\brac{2|v(b)|^2+2(b-a)^2\int_z^b|v'(y)|^2dy}dz\\
&\leq 2(b-a)|v(b)|^2+2(b-a)^3\int_a^b|v'(y)|^2dy.\qed
\end{align*}
\end{proof}

\begin{lemma}
Let $a\geq 2$ and $0<b<c<\infty$. Then for any $v\in C^1([0,c])$ we have
\[\int_0^bz^a|v(z)|^2dz\lesssim_{a,b,c}\int_0^cz^a|v'(z)|^2dz+\int_b^cz^a|v(z)|^2dz.\]
\end{lemma}
\begin{proof}
Let $\phi\in C^\infty([0,\infty))$ be a decreasing function such that $\phi(z)=1$ for $z\leq b$ and $\phi(z)=0$ for $z\geq c$. First note that integration by parts tells us that
\[\int_0^cz^{a-1}\phi v(\phi v)'dz=-\int_0^c(a-1)z^{a-2}(\phi v)^2+z^{a-1}(\phi v)'\phi v\,dz.\]
Using the lemma above, we have
\begin{align*}
\int_0^bz^a|\phi v|^2dz&=\int_0^b\abs{z^{a/2}\phi v}^2dz
\lesssim_c\int_0^c\abs{(z^{a/2}\phi v)'}^2dz\\
&\leq{a^2\over 4}\int_0^cz^{a-2}|\phi v|^2dz+a\int_0^cz^{a-1}\phi v(\phi v)'dz+\int_0^cz^a|(\phi v)'|^2dz\\
&\leq{1\over 4}(a^2-2a(a-1))\int_0^cz^{a-2}|\phi v|^2dz+\int_0^cz^a|(\phi v)'|^2dz\\
&=-{1\over 4}a(a-2)\int_0^cz^{a-2}|\phi v|^2dz+\int_0^cz^a|(\phi v)'|^2dz\\
&\leq\int_0^cz^a|(\phi v)'|^2dz
\end{align*}
Using the fact that $|x+y|^2\leq|x|^2+2|x||y|+|y|^2\leq 2|x|^2+2|y|^2$ we have
\begin{align*}
\int_0^bz^a|v|^2dz&\leq\int_0^bz^a|\phi v|^2dz\lesssim_c\int_0^cz^a|(\phi v)'|^2dz=\int_0^cz^a|\phi'v+\phi v'|^2dz\\
&\leq 2\int_0^cz^a(|\phi' v|^2+|\phi v'|^2)dz
\leq 2\|\phi'\|_\infty^2\int_b^cz^a|v|^2dz+2\int_0^cz^a|v'|^2dz.\qed
\end{align*}
\end{proof}

\begin{proposition}
Let $a\geq 2$. We have
\[\int_0^1z^a|v(z)|^2dz\lesssim_a\int_0^1z^a|v'(z)|^2dz+|v(1)|^2\qquad\text{for all}\qquad v\in C^1([0,1]).\]
\end{proposition}
\begin{proof}
Using the first of the above two lemmas we have
\begin{align*}
\int_{1\over 2}^1z^a|v(z)|^2dz&\leq\int_{1\over 2}^1|v(z)|^2dz\lesssim|v(1)|^2+\int_{1\over 2}^1|v'(y)|^2dy
\leq|v(1)|^2+2^a\int_0^1z^a|v'(z)|^2dz
\end{align*}
And using the second of the above two lemmas we have
\begin{align*}
\int_0^{1\over 2}z^a|v(z)|^2dz\lesssim_a\int_0^1z^a|v'(z)|^2dz+\int_{1\over 2}^1z^a|v(z)|^2dz
\end{align*}
but the rightmost term we have already estimated in the right form. Hence we are done.\qed
\end{proof}

With this Poincar\'e-Hardy-type inequality, we can prove linear stability for stars of small relative central density (see Definition \ref{def 1}) when $\gamma<2(d-1)/d$.

\begin{theorem}
Suppose $\gamma<2(d-1)/d$. There exist $\epsilon>0$ such that the liquid Lane–Emden stars are linearly stable against radial perturbations whenever $\bar\rho(0)-1<\epsilon$ (i.e. small relative central density, Definition \ref{def 1}).
\end{theorem}
\begin{proof}
Let $z=y/R$. Let $\tilde{\chi}(z)=\chi(Rz)$ and $\tilde{\rho}(z)=\bar{\rho}(Rz)$. Then (\ref{TE}) becomes
\begin{align*}
\<L\chi,\chi\>&=\int_0^1\gamma\tilde\rho(z)^\gamma(Rz)^{d+1}{1\over R^2}\partial_z\tilde\chi(z)^2+(2(d-1)-d\gamma)(Rz)^d\tilde\chi(z)^2{1\over R}\partial_z\tilde\rho^\gamma(z)Rdz+d\gamma R^{d}\tilde\chi(1)^2\\
&=R^d\int_0^1\gamma\tilde\rho^\gamma z^{d+1}(\partial_z\tilde\chi)^2+(2(d-1)-d\gamma)z^d\tilde\chi^2\partial_z\tilde\rho^\gamma dz+R^{d}d\gamma\tilde\chi(1)^2.
\end{align*}
We know from our derivation of the existence of steady states that
\begin{align*}
{d\bar\rho^{\gamma-1}\over dy}&=-4\pi{\gamma-1\over\gamma}{1\over y^{d-1}}\int_0^yr^{d-1}\bar\rho(r)dr\geq-4\pi\bar\rho(0){\gamma-1\over\gamma}{1\over y^{d-1}}\int_0^yr^{d-1}dr\\
&=-4\pi\bar\rho(0){\gamma-1\over\gamma}{1\over d}y
\end{align*}
So
\begin{align*}
\partial_y\bar\rho^\gamma&=\gamma\bar\rho^{\gamma-1}\partial_y\bar\gamma={\gamma\over\gamma-1}\bar\rho\partial_y\bar\rho^{\gamma-1}\geq-{4\pi\over d}\bar\rho(0)^2y\\
\partial_z\tilde{\rho}^\gamma&\geq-R{4\pi\over d}\bar\rho(0)^2(Rz)=-{4\pi\over d}\bar\rho(0)^2R^2z.
\end{align*}
So when $2(d-1)-d\gamma\geq 0$ we have
\begin{align*}
\<L\chi,\chi\>&\geq R^d\int_0^1\gamma z^{d+1}(\partial_z\tilde\chi)^2-{4\pi\over d}\bar\rho(0)^2R^2(2(d-1)-d\gamma)z^{d+1}\tilde\chi^2dz+R^{d}d\gamma\tilde\chi(1)^2.
\end{align*}
From the decay estimates we have for the steady states we see that $R\to 0$ as $\bar\rho(0)\searrow 1$. So the above proposition tells us that for small enough relative central density $\bar\rho(0)-1$, we have
\[{4\pi\over d}\bar\rho(0)^2R^2(2(d-1)-d\gamma)\int_0^1z^{d+1}\tilde\chi^2dz<\gamma \int_0^1z^{d+1}(\partial_z\tilde\chi)^2dz+d\gamma\tilde\chi(1)^2\qquad\text{for any}\qquad\tilde{\chi}.\]
It follows that we have stability.\qed
\end{proof}

Finally, it remains to prove linear instability for stars of large central density when $\gamma<2(d-1)/d$.

\begin{theorem}
Suppose $\gamma<2(d-1)/d$ and $d<10$. There exist $C>0$ such that the liquid Lane–Emden stars are linearly unstable whenever $\bar\rho(0)>C$ (i.e. large central density).
\end{theorem}
\begin{proof}
We deal with three sub-cases individually.

\mybb{Case 1: $\gamma>2d/(d+2)$}

We saw that the family of gaseous steady states are self-similar, so that the family is given by $\bar\rho_\kappa(y)=\kappa\bar\rho_*(\kappa^{1-\gamma/2}y)$ where $\bar\rho_*$ is a steady state. The corresponding liquid star has $R=\kappa^{-(1-\gamma/2)}\bar\rho_*^{-1}(1/\kappa)$. So $\tilde{\rho}_\kappa(z)=\kappa\bar\rho_*(\bar\rho_*^{-1}(1/\kappa)z)$. With this, (\ref{TE}) reads
\begin{align*}
\<L_\kappa\chi,\chi\>&=R^d\kappa^\gamma\int_0^1\gamma\bar\rho_*(\bar\rho_*^{-1}(1/\kappa)z)^\gamma z^{d+1}(\partial_z\tilde\chi)^2+(2(d-1)-d\gamma)\bar\rho_*^{-1}(1/\kappa)z^d\tilde\chi^2(\bar\rho_*^\gamma)'(\bar\rho_*^{-1}(1/\kappa)z)dz\\
&\quad+R^{d}d\gamma\tilde\chi(1)^2
\end{align*}

When $\gamma>2d/(d+2)$, the gaseous steady state $\bar\rho_*$ has compact support. Then $\bar\rho_*^{-1}(1/\kappa)\to\bar\rho_*^{-1}(0)=:R_*$ as $\kappa\to\infty$. So
\begin{align*}
\int_0^1 z^d\partial_z\tilde{\rho}_\kappa^\gamma dz=\kappa^\gamma\underbrace{\bar\rho_*^{-1}(1/\kappa)\int_0^1 z^d(\bar\rho_*^\gamma)'(\bar\rho_*^{-1}(1/\kappa)z)dz}_{\to R_*\int_0^1z^d(\bar\rho_*^\gamma)'(R_*z)dz\qquad\text{as}\qquad\kappa\to\infty}
\end{align*}
by dominated convergence where the integrand is dominated by $z^d\|(\bar\rho_*^\gamma)'\|_\infty$. So
\begin{align*}
\<L_\kappa1,1\>&=R^d(2(d-1)-d\gamma)\kappa^\gamma\bar\rho_*^{-1}(1/\kappa)\int_0^1z^d\tilde\chi^2(\bar\rho_*^\gamma)'(\bar\rho_*^{-1}(1/\kappa)z)dz+R^{d}d\gamma\\
&\to-\infty\qquad\text{as}\qquad\kappa\to\infty.
\end{align*}
Hence we have instability for large central density.

\mybb{Case 2: $\gamma=2d/(d+2)$}

From the explicit formula, we have
\begin{align*}
\rho_\kappa^\gamma(y)&=\kappa^{2d\over d+2}\brac{1+{2\pi\over d^2}\kappa^{4\over d+2}y^2}^{-d}=\brac{\kappa^{-{2\over d+2}}+{2\pi\over d^2}\kappa^{2\over d+2}y^2}^{-d}\\
\partial_y\rho_\kappa^\gamma(y)&=-{4\pi\over d}\kappa^{2\over d+2}y\brac{\kappa^{-{2\over d+2}}+{2\pi\over d^2}\kappa^{2\over d+2}y^2}^{-d-1}
\end{align*}
and
\[R_\kappa={d\over\sqrt{2\pi}}\kappa^{-{2\over d+2}}\brac{\kappa^{2\over d+2}-1}^{1\over 2}.\]
Let $\chi_\kappa(y)=\kappa^{{d\over d+2}}\chi(\kappa^{1\over d+2}y)$. From (\ref{TE}) we have
\begin{align*}
\<L_\kappa\chi_\kappa,\chi_\kappa\>
&=\int_0^{R_\kappa}\gamma\bar\rho_\kappa^\gamma y^{d+1}(\partial_y\chi_\kappa)^2+(2(d-1)-d\gamma)y^d\chi_\kappa^2\partial_y\bar\rho_\kappa^\gamma dy+d\gamma R_\kappa^{d}\chi_\kappa(R_\kappa)^2\\
&=\kappa\int_0^{R_\kappa}{2d\over d+2}\brac{\kappa^{-{2\over d+2}}+{2\pi\over d^2}\kappa^{2\over d+2}y^2}^{-d}y^{d+1}\chi'(\kappa^{1\over d+2}y)^2dy\\
&\quad-\kappa{4\pi\over d}\brac{2(d-1)-{2d^2\over d+2}}\int_0^{R_\kappa}y^{d+1}\brac{\kappa^{-{2\over d+2}}+{2\pi\over d^2}\kappa^{2\over d+2}y^2}^{-d-1}\chi(\kappa^{1\over d+2}y)^2dy\\
&\quad+d\gamma\kappa^{{d\over d+2}}R_\kappa^{d}\chi(\kappa^{1\over d+2}R_\kappa)^2\\
&={2d\over d+2}\int_0^{\kappa^{1\over d+2}R_\kappa}\brac{\kappa^{-{2\over d+2}}+{2\pi\over d^2}z^2}^{-d}z^{d+1}\chi'(z)^2dz\\
&\quad-{8\pi\over d}{d-2\over d+2}\int_0^{\kappa^{1\over d+2}R_\kappa}\brac{\kappa^{-{2\over d+2}}+{2\pi\over d^2}z^2}^{-d-1}z^{d+1}\chi(z)^2dz\\
&\quad+{2d^2\over d+2}\kappa^{{d\over d+2}}R_\kappa^{d}\chi(\kappa^{1\over d+2}R_\kappa)^2\\
&\to{2d\over d+2}\brac{d^2\over 2\pi}^{d}\int_0^{d\over\sqrt{2\pi}}z^{-d+1}\chi'(z)^2dz
-{8\pi\over d}{d-2\over d+2}\brac{d^2\over 2\pi}^{d+1}\int_0^{d\over\sqrt{2\pi}}z^{-d-1}\chi(z)^2dz\\
&\quad+{2d^2\over d+2}\brac{d^2\over 2\pi}^{d\over 2}\chi\brac{d\over\sqrt{2\pi}}^2\qquad\text{as}\qquad\kappa\to\infty
\end{align*}
Let $\chi(z)=1$, then we see that $\<L_\kappa\chi_\kappa,\chi_\kappa\>\to-\infty$ as $\kappa\to\infty$. Hence we have instability.

\mybb{Case 3: $\gamma<2d/(d+2)$}

From (\ref{TE}) we have
\begin{align*}
\<L_\kappa\chi,\chi\>
&=\int_0^{R_\kappa}\gamma\bar\rho_\kappa^\gamma y^{d+1}(\partial_y\chi)^2+(2(d-1)-d\gamma)y^d\chi^2\partial_y\bar\rho_\kappa^\gamma dy+d\gamma R_\kappa^{d}\chi(R_\kappa)^2\\
&=\int_0^{R_\kappa}\gamma\bar\rho_\kappa^\gamma y^{d+1}(\partial_y\chi)^2-(2(d-1)-d\gamma)y\chi^2\bar m_\kappa\bar\rho_\kappa dy+d\gamma R_\kappa^{d}\chi(R_\kappa)^2.
\end{align*}
Fix $\delta>0$, and suppose $\chi$ is constant on $[0,\epsilon]$. Then we have
\begin{align*}
\<L_\kappa\chi,\chi\>&\leq\int_\epsilon^{R_\kappa}\gamma\bar\rho_\kappa^\gamma y^{d+1}(\partial_y\chi)^2-(2(d-1)-d\gamma)y\chi^2\bar m_\kappa\bar\rho_\kappa dy+d\gamma R_\kappa^{d}\chi(R_\kappa)^2\\
&\to\int_\epsilon^{R_\infty}\gamma(v^*_1)^\gamma y^{d+1-{2\gamma\over 2-\gamma}}(\partial_y\chi)^2-(2(d-1)-d\gamma)y^{d+1-{4\over 2-\gamma}}\chi^2v^*_2v^*_1dy+d\gamma R_\infty^{d}\chi(R_\infty)^2\\
&=\gamma(v^*_1)^\gamma\int_\epsilon^{R_\infty}y^{d+1-{2\gamma\over 2-\gamma}}(\partial_y\chi)^2-2\brac{d-{2\over 2-\gamma}}y^{d-{2+\gamma\over 2-\gamma}}\chi^2dy+d\gamma R_\infty^{d}\chi(R_\infty)^2
\end{align*}
as $\kappa\to\infty$. Let $\chi(y)=\epsilon^{-a}\wedge y^{-a}$. Then
\begin{align*}
&\gamma(v^*_1)^\gamma\int_\epsilon^{R_\infty}y^{d+1-{2\gamma\over 2-\gamma}}(\partial_y\chi)^2-2\brac{d-{2\over 2-\gamma}}y^{d-{2+\gamma\over 2-\gamma}}\chi^2dy+d\gamma R_\infty^{d}\chi(R_\infty)^2\\
&=\gamma(v^*_1)^\gamma\brac{a^2-2\brac{d-{2\over 2-\gamma}}}\int_\epsilon^{R_\infty}y^{d-{2+\gamma\over 2-\gamma}-2a}dy+d\gamma R_\infty^{d-2a}
\end{align*}
Let $a={1\over 2}(d-{2\gamma\over 2-\gamma})$ so that ${d-{2+\gamma\over 2-\gamma}-2a}=-1$. By choosing $\epsilon$ small enough we can make the integral large enough in magnitude. So if we have instability if
\[\brac{a^2-2\brac{d-{2\over 2-\gamma}}}<0.\]
We have
\[\brac{d/2-a}(2-\gamma)=\gamma\qquad\iff\qquad\gamma={d-2a\over 1+d/2-a}\]
So the above condition is
\begin{align*}
0>a^2-2\brac{d-(1+d/2-a)}=a^2-2a-(d-2)
\end{align*}
By definition $a\leq{1\over 2}(d-2)$, so it suffice to have
\[a^2-2a-(d-2)\leq a^2-4a<0\]
$a^2-4a$ is negative when $a\in(0,4)$. Hence we are done if ${1\over 2}(d-2)<4$, or equivalently $d<10$.\qed
\end{proof}

This completes the proof of theorem \ref{liquid Lane–Emden stars stability}.

\section*{Acknowledgements}

The author express his gratitude to M. Had\v{z}i\'c for helpful discussions. K. M. Lam acknowledges the support of the EPSRC studentship grant EP/R513143/1.

\appendix

\section{Standard results for gaseous stars}

\begin{lemma}
Let $x,y\in[a,b]$ where $0<a<b$.
\begin{enumerate}
\item If $\alpha\in[0,1]$, then $|x^\alpha-y^\alpha|\leq\alpha a^{\alpha-1}|x-y|$.
\item If $\alpha\in[1,\infty)$, then $|x^\alpha-y^\alpha|\leq\alpha b^{\alpha-1}|x-y|$.
\end{enumerate}
\end{lemma}
\begin{proof}
This follows from the mean value inequality apply to the function $f(x)=x^\alpha$. In case i, $|f'(x)|$ is bounded by $\alpha a^{\alpha-1}$ on $[a,b]$. In case ii, $|f'(x)|$ is bounded by $\alpha b^{\alpha-1}$ on $[a,b]$.\qed
\end{proof}

\begin{lemma}
Let $x,y\in\R$, then
\[|e^x-e^y|\leq\max\{e^x,y^y\}|x-y|.\]
\end{lemma}
\begin{proof}
This follows from the mean value inequality apply to the function $f(x)=e^x$, noting that $|f'(x)|$ is bounded by $\max\{e^x,y^y\}$ on $[x,y]$.\qed
\end{proof}

\begin{theorem}
For every $\rho_0>0$, the steady state equation admits a unique solution $\rho\geq 0$ such that $\rho(0)=\rho_0$. The interval of existence $[0,R)$ is such that either $R=\infty$ or $\lim_{r\to R}\rho(r)=0$.
\end{theorem}
\begin{proof}

\mybb{When $\gamma>1$:}

(\ref{*}) is equivalent to
\[w(r)=w_0-4\pi{\gamma-1\over\gamma}\int_0^r{1\over z^{d-1}}\int_0^zy^{d-1}w(y)^\alpha dydz:=T(w)(r).\]
Assume $w_0>\epsilon>0$, we claim that for small enough $\delta>0$, $T$ maps $C([0,\delta],[\epsilon,w_0])$ to itself. Note that for $r\in[0,\delta]$ and $w\in C([0,\delta],[\epsilon,w_0])$ we have
\begin{align*}
0&\leq 4\pi{\gamma-1\over\gamma}\int_0^r{1\over z^{d-1}}\int_0^zy^{d-1}w(y)^\alpha dydz\\
&\leq 4\pi{\gamma-1\over\gamma}w_0\int_0^r{1\over z^{d-1}}\int_0^zy^{d-1}dydz
=4\pi{\gamma-1\over\gamma}w_0\int_0^r{1\over z^{d-1}}{1\over d}z^ddz
={2\pi\over d}{\gamma-1\over\gamma}w_0\delta^2
\end{align*}
Choosing $\delta$ small enough we can make this smaller than $w_0-1$. It follows that $T(w)\in C([0,\delta],[\epsilon,w_0])$. Now for $u,v\in C([0,\delta],[\epsilon,w_0])$ we have using the first lemma
\begin{align*}
\|T(u)-T(v)\|_{C([0,\delta])}&\leq 4\pi{\gamma-1\over\gamma}\sup_{r\in[0,\delta]}\abs{\int_0^r{1\over z^{d-1}}\int_0^zy^{d-1}(u(y)^\alpha-v(y)^\alpha)dydz}\\
&\leq{2\pi\over d}{\gamma-1\over\gamma}\delta^2\max\{\alpha\epsilon^{\alpha-1},\alpha w_0^{\alpha-1}\}\|u-v\|_{C([0,\delta])}
\end{align*}
By shrinking $\delta>0$ further, we can make $T$ a contraction map. Hence a unique fixed point for $T$ exist. It follows that our equation has a unique solution in a small interval $[0,\delta]$. For $r_0>0$, we claim that as long as $w(r_0)>0$, we can extend the solution beyond $r_0$. Indeed, our equation
\[{1\over r^{d-1}}{d\over dr}\brac{r^{d-1}{dw\over dr}}=-4\pi{\gamma-1\over\gamma}w^\alpha\]
is equivalent to
\begin{align*}
w(r)&=w_{r_0}+\int_{r_0}^r{1\over z^{d-1}}\brac{r_0^{d-1}w'_{r_0}-4\pi{\gamma-1\over\gamma}\int_{r_0}^zy^{d-1}w(y)^\alpha dy}dz\\
&=w_{r_0}+{r_0^{d-1}w'_{r_0}\over d-2}\brac{{1\over r_0^{d-2}}-{1\over r^{d-2}}}-4\pi{\gamma-1\over\gamma}\int_{r_0}^r{1\over z^{d-1}}\int_{r_0}^zy^{d-1}w(y)^\alpha dydz:=H(w)(r)
\end{align*}
Fix $\epsilon>0$ such that $\epsilon<w_{r_0}$. A similar computation as above show that for small enough $\delta>0$, $T$ maps $C(\bar B_\delta(r_0),[\epsilon,w_0])$ to itself. Furthermore,
\begin{align*}
\|H(u)-H(v)\|_{C(\bar B_\delta(r_0))}&\leq 4\pi{\gamma-1\over\gamma}\sup_{r\in[0,\delta]}\abs{\int_{r_0}^r{1\over z^{d-1}}\int_{r_0}^zy^{d-1}(u(y)^\alpha-v(y)^\alpha)dydz}\\
&\leq 4\pi{\gamma-1\over\gamma}\delta^2\brac{r_0+\delta\over r_0-\delta}^{d-1}\max\{\alpha\epsilon^{\alpha-1},\alpha w_0^{\alpha-1}\}\|u-v\|_{C(\bar B_\delta(r_0))}
\end{align*}
By shrinking $\delta>0$ further, we can make $H$ a contraction map. Hence a unique fixed point for $H$ exist. It follows that our equation has a unique solution on small interval $\bar B_\delta(r_0)$. The overlapping bit with the previous solution agrees, so we have extended our solution.

It follows that a solution our equation exist on a maximal interval $[0,R)$ such that either $R=\infty$ or $\lim_{r\to R}w(r)=0$, noting that $w$ is a deceasing function because $w'$ is always non-positive as can be seen from (\ref{*}).

\mybb{When $\gamma>1$:}

This time we consider the (\ref{$h$-equation}) and (\ref{**}). Using the second lemma instead of the first, the same proof of existence carries over, without having to cap $h$ at 0. It follows that our equation exist on a maximal interval $[0,R)$ such that either $R=\infty$ or $\lim_{r\to R}h(r)=-\infty$.\qed
\end{proof}

\begin{theorem}
Suppose $w$ is a gas star. Then $w$ has compact support if
\[\gamma>{2d\over d+2}\qquad\text{ or equivalently }\qquad\alpha<{d+2\over d-2}\]
and infinitely support otherwise.
\end{theorem}
\begin{proof}
When $\gamma=1$, (\ref{**}) gives
\[h'(r)\geq-4\pi{1\over r^{d-1}}e^{h_0}\int_0^ry^{d-1}dy=-{4\pi\over d}e^{h_0}r.\]
So
\[h(r)\geq h_0-{4\pi\over d}e^{h_0}\int_0^rydy=h_0-{2\pi\over d}e^{h_0}r^2.\]
So the gaseous star cannot have compact support.

We will prove the $\gamma>1$ case in three steps:

\mybb{Step 1:}

If the gaseous solution $w$ has compact support, then the Pohozaev integral gives
\begin{align*}
2\pi{\gamma-1\over\gamma}\brac{{2d\over 1+\alpha}-(d-2)}\int_0^Rw^{\alpha+1}y^{d-1}dy
={1\over 2}w'(R)^2R^d>0
\end{align*}
This means $w$ cannot have compact support if
\begin{align*}
{2d\over 1+\alpha}\leq d-2\qquad\iff\qquad\alpha\geq{d+2\over d-2}\qquad\iff\qquad\gamma\leq{d-2\over d+2}+1={2d\over d+2}.
\end{align*}

\mybb{Step 2:}

Now suppose the condition in the proposition hold but $w$ has infinite support. Fix some $\epsilon>0$. From (\ref{*}) we have
\[w'(r)\leq-{1\over r^{d-1}}\underbrace{4\pi{\gamma-1\over\gamma}\int_0^\epsilon y^{d-1}w(y)^\alpha dy}_{:=m_\epsilon}\qquad\text{for}\qquad r\geq\epsilon.\]
So
\[w(r)=-\int^\infty_rw'(y)dy\geq m_\epsilon\int^\infty_r{1\over y^{d-1}}dy={m_\epsilon\over d-2}{1\over r^{d-2}}.\]
Combining this and the decay estimate for $w$ we get
\[{2\pi\over d}{\gamma-1\over\gamma}\leq\begin{cases}\displaystyle{1\over\alpha-1}\brac{{1\over w(r)^{\alpha-1}}-{1\over w_0^{\alpha-1}}}\brac{d-2\over m_\epsilon}^{2\over d-2}w(r)^{2\over d-2}&\qquad\text{ when }\qquad\alpha\not=1\\
\displaystyle\brac{\ln w_0-\ln w(r)}\brac{d-2\over m_\epsilon}^{2\over d-2}w(r)^{2\over d-2}&\qquad\text{ when }\qquad\alpha=1\end{cases}.\]
If
\[{2\over d-2}-(\alpha-1)>0\qquad\text{ or equivently }\qquad(\alpha-1)(d-2)<2,\]
then the RHS tends to 0 as $r\to\infty$ but not the LHS, this is a contradiction. So $R<\infty$. So we know there exist compactly supported gas solutions at least when $(\alpha-1)(d-2)<2$. 

\mybb{Step 3:}

Last step, we prove compact support for when $(\alpha-1)(d-2)<2$. So we can assume now
\[(\alpha-1)(d-2)\geq 2\qquad\iff\qquad d\geq{2\alpha\over\alpha-1}.\]
In particular $\alpha>1$. Note that by the decay estimate we must have
\[w(r)^{\alpha-1}r^{\sigma}\to 0\qquad\text{as}\qquad r\to\infty\qquad\text{for any}\qquad\sigma\in[0,2).\]
The decay estimate also gives us
\[w(r)\leq\brac{w_0^{-(\alpha-1)}+(\alpha-1){2\pi\over d}{\gamma-1\over\gamma}r^2}^{-{1\over\alpha-1}}\leq\brac{(\alpha-1){2\pi\over d}{\gamma-1\over\gamma}}^{-{1\over\alpha-1}}r^{-{2\over\alpha-1}}.\]
Now using $(*)$, we can bound $w'$ as follows
\begin{align*}
0\geq w'(r)&=-4\pi{\gamma-1\over\gamma}{1\over r^{d-1}}\int_0^ry^{d-1}w(y)^\alpha dy
\gtrsim-{1\over r^{d-1}}\brac{w_0^\alpha+\int_1^ry^{d-1-{2\alpha\over\alpha-1}}dy}\\
&\gtrsim\begin{cases}\displaystyle-{C+r^{d-{2\alpha\over\alpha-1}}\over r^{d-1}}\\\displaystyle-{C+\ln r\over r^{d-1}}\end{cases}\gtrsim\begin{cases}\displaystyle-r^{-(d-1)}-r^{1-{2\alpha\over\alpha-1}}&\displaystyle\qquad\text{ when }\qquad d\not={2\alpha\over\alpha-1}\\-r^{-(d-1)}-r^{-(d-1)}\ln r&\displaystyle\qquad\text{ when }\qquad d={2\alpha\over\alpha-1}\end{cases}
\end{align*}
where $C$ is some constant. Recall we assume in this step
\[d\geq{2\alpha\over\alpha-1}\qquad\iff\qquad d-1\geq{2\alpha\over\alpha-1}-1.\]
So the estimate means that
\[w'(r)r^{\sigma}\to 0\qquad\text{as}\qquad r\to\infty\qquad\text{for any}\qquad\sigma\in\bigg[0,{2\alpha\over\alpha-1}-1\bigg).\]
Using these estimates, we see that $w'(r)^2r^d\to 0$ as $r\to\infty$ if
\[{4\alpha\over\alpha-1}-2>d\qquad\iff\qquad 4\alpha-2\alpha+2>d\alpha-d\qquad\iff\qquad\alpha<{d+2\over d-2}.\]
Also $w(r)^{\alpha+1}r^d\to 0$ as $r\to\infty$ if
\[2{\alpha+1\over\alpha-1}>d\qquad\iff\qquad 2\alpha+2>d\alpha-d\qquad\iff\qquad\alpha<{d+2\over d-2}.\]
And finally, $w'(r)w(r)r^{d-1}\to 0$ as $r\to\infty$ if
\[{2\alpha\over\alpha-1}-1+{2\over\alpha-1}>d-1\qquad\iff\qquad 2{\alpha+1\over\alpha-1}>d\qquad\iff\qquad\alpha<{d+2\over d-2}.\]
Since we assumed $w$ has infinite support, the Pohozaev integral holds for $r\in[0,\infty)$. Taking limit $r\to\infty$ in the Pohozaev integral then gives
\[2\pi{\gamma-1\over\gamma}\brac{{2d\over 1+\alpha}-(d-2)}\int_0^\infty w^{\alpha+1}y^{d-1}dy=0.\]
But the LHS must be strictly positive, this is a contraction. Hence $w$ must be compactly supported.\qed
\end{proof}

\begin{proposition}
When $\gamma=2d/(d+2)$ we have explicit steady state solution
\[w(r)=A\brac{1+{2\pi\over d^2}A^{4\over d-2}r^2}^{1-d/2}\quad\text{or equivalently}\quad\rho(r)=C\brac{1+{2\pi\over d^2}C^{4\over d+2}r^2}^{-1-d/2}.\]
And the support of the liquid star is
\begin{align*}
R=\brac{{d^2\over 2\pi}C^{-{4\over d+2}}(C^{2\over d+2}-1)}^{1\over 2}.
\end{align*}
\end{proposition}
\begin{proof}
Consider
\begin{align*}
w(r)&=A(1+Br^b)^a\\
w'(r)&=ABabr^{b-1}(1+Br^b)^{a-1}\\
w''(r)&=ABab(b-1)r^{b-2}(1+Br^b)^{a-1}+AB^2ab^2(a-1)r^{2(b-1)}(1+Br^b)^{a-2}
\end{align*}
Substitute into the steady state equation
\[w''+(d-1)r^{-1}w'=-4\pi{\gamma-1\over\gamma}w^\alpha\]
we get
\begin{align*}
&-4\pi{\gamma-1\over\gamma}A^\alpha(1+Br^b)^{\alpha a}\\
&=ABab(b-1+d-1)r^{b-2}(1+Br^b)^{a-1}+AB^2ab^2(a-1)r^{2(b-1)}(1+Br^b)^{a-2}\\
&=ABab\brac{(b-1+d-1)r^{b-2}+B(b(a-1)+(b-1+d-1))r^{2(b-1)}}(1+Br^b)^{a-2}
\end{align*}
For this to have the possibility to hold, we need $b=2$ and $a=1-d/2$. Then the equation becomes
\begin{align*}
-4\pi{\gamma-1\over\gamma}A^\alpha(1+Br^2)^{\alpha(1-d/2)}
=ABd(2-d)(1+Br^2)^{-(1+d/2)}
\end{align*}
For this to have the possibility to hold, we need
\[{1\over\gamma-1}=\alpha=-{2+d\over 2-d}={d+2\over d-2}\qquad\iff\qquad \gamma={d-2\over d+2}+1={2d\over d+2}.\]
And then we also need
\[-4\pi{d-2\over 2d}A^{d+2\over d-2}=ABd(2-d)\qquad\iff\qquad 2\pi A^{4\over d-2}=Bd^2\qquad\iff\qquad B={2\pi\over d^2}A^{4\over d-2}.\]
So when $\gamma=2d/(d+2)$ we have explicit steady state solution
\[w(r)=A\brac{1+{2\pi\over d^2}A^{4\over d-2}r^2}^{1-d/2}\quad\text{or equivalently}\quad\rho(r)=C\brac{1+{2\pi\over d^2}C^{4\over d+2}r^2}^{-1-d/2}.\]
The support $R$ of the liquid star is then given by
\begin{align*}
1=C\brac{1+{2\pi\over d^2}C^{4\over d+2}R^2}^{-1-d/2}&\qquad\iff\qquad 1+{2\pi\over d^2}C^{4\over d+2}R^2=C^{2\over d+2}\\
&\qquad\iff\qquad R=\brac{{d^2\over 2\pi}C^{-{4\over d+2}}(C^{2\over d+2}-1)}^{1\over 2}.\qed
\end{align*}
\end{proof}

\begin{proposition}[Self-similarity of solutions]
Let $\rho$ be a gaseous steady state. Then $\rho_\kappa(r)=\kappa\rho(\kappa^{1-\gamma/2}r)$ is a gaseous steady state for any $\kappa>0$, and the corresponding liquid star has support $R=\kappa^{-(1-\gamma/2)}\rho^{-1}(1/\kappa)$.
\end{proposition}
\begin{proof}
First we deal with the $\gamma>1$ case. Recall the steady state equation
\[w''(r)+(d-1)r^{-1}w'(r)=-4\pi{\gamma-1\over\gamma}w(r)^\alpha.\]
Let $v(r)=\kappa w(\kappa^\beta r)$. Then $v'(r)=\kappa^{\beta+1}w'(\kappa^\beta r)$ and $v''(r)=\kappa^{2\beta+1}w''(\kappa^\beta r)$. So
\begin{align*}
v''(r)+(d-1)r^{-1}v'(r)&=\kappa^{2\beta+1}w''(\kappa^\beta r)+(d-1)(\kappa^\beta r)^{-1}\kappa^{2\beta+1}w'(\kappa^\beta r)\\
&=-\kappa^{2\beta+1}4\pi{\gamma-1\over\gamma}w(\kappa^\beta r)^\alpha
=-\kappa^{2\beta+1-\alpha}4\pi{\gamma-1\over\gamma}v(r)^\alpha
\end{align*}
So if we choose
\[\beta={1\over 2}(\alpha-1)={1\over 2}{2-\gamma\over\gamma-1},\]
then $v$ is again a solution to the steady state equation. Converting back to $\rho$ gives the desired result.

For $\gamma=1$, it is clear by substitution that $\rho_\kappa(r)=\kappa\rho(\sqrt\kappa r)$ is a solution to the steady state equation
\[0=\lpc(\ln\rho)+4\pi\rho={\rho''\over\rho}-{(\rho')^2\over\rho^2}+(d-1){1\over r}{\rho'\over \rho}+4\pi\rho.\qed\]
\end{proof}

\end{document}